\documentclass[11pt]{article}
\usepackage[mathscr]{eucal}
\usepackage{amsthm,amsmath,cite}

\DeclareMathAlphabet{\mathpzc}{OT1}{pzc}{m}{it}

\newcommand{{\M}}{{\textsf{{L}$_\sharp^{2}$}}}

\newcommand{\A}{\mathbb A}
\newcommand{\B}{\mathbb B}

\newcommand{\bsfV}{\boldsymbol{\mathsf V}}

\usepackage{bbm}
\usepackage{dsfont}
\usepackage{color}
\usepackage{mathrsfs}
\usepackage{amssymb}
\usepackage{calligra}
\usepackage{fontenc}
\usepackage{amsbsy}
\usepackage{comment}

\newcommand{\bfsf}[1]{{\textbf{\textsf{#1}}}}

\usepackage{graphicx}
\graphicspath{%
    {converted_graphics/}
    {/}
}

\newcommand{\Div}{\mbox{\rm div}\,}

\newcommand{\Int}[2]{{\displaystyle \int_{ #1}^{ #2}}}
\newcommand{\Lim}[1]{{\displaystyle \lim_{ #1}}}

\newcommand{\BCup}[2]{{\displaystyle \bigcup_{#1}^{#2}}}

\newcommand{\Frac}[2]{\displaystyle{\frac{\displaystyle{#1}}{\displaystyle{#2}}}}

\newcommand{\beea}{\begin{eqnarray}}
\newcommand{\eeea}{\end{eqnarray}}

\newcommand{\bfe}{{\mbox{\boldmath $e$}} }

\newcommand{\bfz}{{\mbox{\boldmath $z$}} }
\newcommand{\0}{{\mbox{\boldmath $0$}} }

\newcommand{\BF}{\begin{footnotesize}}
\newcommand{\EF}{\end{footnotesize}}
\newcommand{\ode}[2]{{\displaystyle \frac{\mbox{$d #1$}}{\mbox{$d #2$}}}}

\newcommand{\bi}{\begin{itemize}}
\newcommand{\ei}{\end{itemize}}
\newcommand{\ed}{\end{document}}
\newcommand{\be}{\begin{equation}}
\newcommand{\ba}{\begin{array}}
\newcommand{\ea}{\end{array}}
\newcommand{\ee}{\end{equation}}
\newcommand{\eeq}[1]{\label{eq:#1}\end{equation}}
\newcommand{\real}{{\mathbb R}}

\newcommand{\bfpsi}{\mbox{\boldmath $\psi$}}

\newcommand{\bfrho}{\mbox{\boldmath $\rho$}}

\newcommand{\bfx}{\mbox{\boldmath $x$}}
\newcommand{\bfy}{\mbox{\boldmath $y$}}

\newcommand{\bfxi}{\mbox{\boldmath $\xi$}}

\newcommand{\bfphi}{\mbox{\boldmath $\varphi$}}

\newcommand{\bfv}{{\mbox{\boldmath $v$}} }
\newcommand{\bfu}{{\mbox{\boldmath $u$}} }

\newcommand{\bfw}{{\mbox{\boldmath $w$}} }

\newcommand{\bfa}{{\mbox{\boldmath $a$}} }

\newcommand{\bfR}{{\mbox{\boldmath $R$}} }

\newcommand{\bfG}{{\mbox{\boldmath $G$}} }

\newcommand{\bfH}{{\mbox{\boldmath $H$}} }
\newcommand{\bfI}{{\mbox{\boldmath $I$}} }

\newcommand{\cald}{{\cal D}}

\newcommand{\cali}{{\cal I}}

\newcommand{\bfSigma}{\mbox{\boldmath $\Sigma$}}

\newcommand{\bfeta}{\mbox{\boldmath $\eta$}}

\newcommand{\bfV}{{\mbox{\boldmath $V$}} }
\newcommand{\bfU}{{\mbox{\boldmath $U$}} }

\newcommand{\bfb}{{\mbox{\boldmath $b$}} }

\newcommand{\bfn}{{\mbox{\boldmath $n$}} }

\def\Bbb R{\real}
\def\hat{\widehat}
\def\tilde{\widetilde}
\def\bar{\overline}

\newcommand{\bfchi}{\mbox{\boldmath $\chi$}}

\newcommand{\bfgamma}{\mbox{\boldmath $\gamma$}}
\newcommand{\bfdelta}{\mbox{\boldmath $\delta$}}

\newcommand{\ED}{\end{description}}

\newtheorem{definition}{Definition}[section]
\newcommand{\Bd}{\begin{definition}\begin{rm}}
\newcommand{\Ed}{\end{rm}\end{definition}}
\newtheorem{remark}{Remark}[section]
\newcommand{\Br}{\begin{remark}\begin{rm}}
\newcommand{\Er}{\end{rm}\end{remark}}
\newtheorem{proposition}{Proposition}[section]
\newcommand{\Bp}{\begin{proposition}\begin{sl}}
\newcommand{\EP}[1]{\end{sl}\label{proposition:#1}\end{proposition}}

\newcommand{\Bt}{\begin{theorem}\begin{sl}}
\newcommand{\Et}{\end{sl}\end{theorem}}
\newcommand{\Bl}{\begin{lemma}\begin{sl}}
\newcommand{\El}{\end{sl}\end{lemma}}
\newtheorem{theorem}{Theorem}[section]
\newtheorem{lemma}{Lemma}[section]

\newtheorem{corollary}{Corollary}[section]
\renewcommand{\eqref}[1]{{\rm (\ref{eq:#1})}}
\newcommand{\Bc}{\begin{corollary}\begin{sl}}
\newcommand{\Ec}{\end{sl}\end{corollary}}
\newcommand{\ET}[1]{\end{sl}\label{theorem:#1}\end{theorem}}
\newcommand{\EDD}[1]{\end{rm}\label{definition:#1}\end{definition}}
\newcommand{\EL}[1]{\end{sl}\label{lemma:#1}\end{lemma}}

\newcommand{\ER}[1]{\end{rm}\label{remark:#1}\end{remark}}
\newcommand{\EC}[1]{\end{sl}\label{corollary:#1}\end{corollary}}

\numberwithin{equation}{section}
\usepackage{enumerate}
\begin{document}

\title{Forced Oscillations of a   Spring-Mounted Body by a Viscous Liquid:\\  Rotational Case} 
\author{Denis Bonheure\thanks{D\'epartement de Math\'ematique, Universit\'e Libre de Bruxelles, Belgium}\,,
\  Giovanni P. Galdi\thanks{Department of Mechanical Engineering and Materials Science,University of Pittsburgh, USA.}
\  \&\,\ Clara Patriarca$^{*}$}
\date{}
\maketitle\noindent
	\noindent
\begin{abstract}
We study the periodic motions of the coupled system $\mathscr S$, consisting of an incompressible Navier-Stokes fluid interacting with a structure formed by a rigid body subject to {\em undamped} elastic restoring forces and torque around its rotation axis. The motion of $\mathscr S$ is driven by the uniform flow of the liquid, far away from the body, characterized by a time-periodic velocity field, $\bfV$, of frequency $f$. We show that the corresponding set of governing equations always possesses a time-periodic weak solution of the same frequency $f$, whatever $f>0$, the magnitude of $\bfV$ and the values of physical parameters. Moreover, we show that the amplitude of linear and rotational displacement is always pointwise in time uniformly bounded by one and the same constant depending on the data, regardless of whether $f$ is or is not close to a natural frequency of the structure. Thus, our result rules out the occurrence of resonant phenomena. 

\end{abstract}

\section{Introduction}
As is well known, one of the central topics of research in the field of fluid-solid interactions is the study of the oscillations of an elastic structure due to the action of a fluid in a time-periodic regime. In fact, this  kind of problems occurs in a wide range of fundamental applications at different scales, including aeroelasticity \cite{Conner}, towing of airborne or underwater bodies \cite{WI}, suspension bridges \cite{diana}, and physiological flows \cite{HH}. We refer to the comprehensive monograph \cite{Bev} and the references therein for more details. Among the many lines of investigation, and especially for its relevance to the integrity of the structure, the one concerning the phenomenon of resonance is particularly significant. This phenomenon occurs when the frequency of the incident fluid flow becomes increasingly close to one, or a multiple, of the natural frequencies of the structure. In this case, the amplitude of the oscillation of the latter can increase dramatically and lead, in some cases, to serious damage or even collapse of the structure.
\par 
In recent years, we have started a systematic study of this phenomenon from a rigorous mathematical point of view, aimed, in particular, at understanding  wind-induced resonance in suspension bridges \cite{BeBoGaGaPe,BoGa,BoGaGa1,BoGaGa2,BoGaGa3,bonheure2024longtime,Pa}. The model we used is the classic one proposed by engineers, see e.g. \cite{Bev}, where the structure, i.e. the cross-section of the bridge, is a rigid body subject to a linear elastic restoring force exerted by the suspenders, while the wind motion is described by the Navier-Stokes equations driven by a uniform flow at spatial infinity, characterized by a constant or time-periodic velocity field. Our analysis was performed  in different situations, namely, when the oscillation of the structure is due to the time periodic flow generated either by spontaneous Hopf bifurcation from a 
steady-state \cite{BoGaGa3} or by a time-periodic uniform flow at large spatial distances \cite{BoGa}. In both cases, we found out that the coupled  fluid-structure system of equations possesses  time-periodic ($T$-periodic) solutions  corresponding to data of {\em arbitrarily given} frequency, and therefore also coinciding with the natural ones of the structure. For this reason we conclude that, at least within these models, resonance cannot occur.  
\par
It must be remarked, however, that the results in \cite{BoGa,BoGaGa3}, although rather interesting on the one hand, are, on the other hand, somewhat still unsatisfactory. In fact, they provide no control over the amplitude of the structure's oscillations, which, in principle, could therefore become unlimitedly large in the vicinity of the natural frequency. Furthermore, in the model described above, the bridge is not allowed to rotate around its longitudinal axis, while rotation, which adds an extra degree of freedom to the structure, could actually be an important factor in causing resonance. 
\par
The primary objective of this paper is to investigate the phenomenon of resonance by taking into account  these two aspects. Specifically, we consider a more realistic model than in \cite{BoGa}, namely the body is now allowed to rotate along a given direction,  subject to the restoring torque exerted by the suspenders. To address the best case scenario that could favor resonance, we assume that there is {\em no damping mechanism}  associated with the torque and that the {\em only} source of dissipation for the coupled fluid-structure system comes from the viscosity of the fluid. We further  suppose that the motion of the system is driven by a uniform, $T$-periodic velocity field, $\bfV$, of the fluid at large distance (spatial infinity) from the body. Under these circumstances we then show that, for any $T>0$, any (sufficiently regular) $\bfV$ and any values of the physical parameters entering the problem,  the corresponding governing set of equations admits at least one $T$-periodic weak solution; we refer to Theorem~\ref{main:th} for the precise statement. Furthermore, improving on previous work, we are able to obtain {\em uniform and pointwise estimates} of the amplitude of the linear and angular oscillations of the body in terms of data showing, in particular, that such oscillations are entirely controlled by the magnitude of $\bfV$ and are {\em independent} of how close the frequency $2\pi/T$ may be to one of the natural frequencies of the structure; see  \eqref{bound_weaksol}, \eqref{bound_point}. Our results thus imply that the resonance phenomenon is absent and that the viscosity of the fluid alone, no matter how small, is sufficient to prevent it from occurring.   
\par
The method we employ is a non-trivial adaptation of the one introduced in \cite{BoGa} to the more general situation considered here, in combination with the new ideas needed to derive the aforementioned pointwise control of the linear and angular oscillations of the body.  In this regard, it should be emphasized that, in the case in question, there are further aspects, due precisely to the additional angular degree of freedom, which render the problem even more interesting. Actually, in order to make the region of flow independent of time {\em and} known, we rewrite --as customary-- the governing equations in a frame attached to the body; see \eqref{01-translation+rotation_or}.   In doing so, however, the direction of $\bfV$ becomes a {\em further} unknown. Moreover,  the classical procedure of lifting $\bfV$ by shifting the flow velocity with a compactly supported extension is no longer feasible; see the comments at the beginning of Section \ref{sec:body-fixed}. 
\par
Our approach is developed in several steps. Therefore, in order not to obscure the guiding idea, we provide, for the reader's convenience, a detailed description of them sequentially, just after the statement of our main result, given in Theorem \ref{main:th}. There, we also outline the corresponding main challenges.
\par
The plan of the paper is as follows. In Section~\ref{sec:form} we begin to give the mathematical formulation of the problem along with its reformulation in a body-fixed frame, see Section~\ref{sec:body-fixed}. Successively, in Section~\ref{Sec:main}, after introducing some relevant function classes, we provide the definition of $T$-periodic weak solution and state our main result in Theorem~\ref{main:th}. The next Section~\ref{sec:IVBP} is dedicated to the proof of well-posedness of the initial-boundary value problem for the suitably mollified governing equations in the bounded domain $\Omega_R:=\Omega\cap \{|x|<R\}$,  where $R>0$ is arbitrary and sufficiently large; see Lemma~\ref{Lemma_glob_ex}. In Section~\ref{Sec:Haraux_trick} we show the crucial result that the total energy, kinetic and potential, associated to the solutions found in the previous sections is exponentially decreasing, at least in the time interval $[0,T]$. This will allow us 
to show existence of $T$-periodic solutions in $\Omega_R$ in Section~\ref{Sec:VeNa}, and then furnish, in Section \ref{Sec8}, the full proof of the main result Theorem~\ref{main:th}. It should be finally observed that the reformulation in the body-fixed frame is only necessary if the body possesses no rotational symmetry around its axis of rotation. In case it does, proofs can be fairly simplified, as explained in the final Section~\ref{sec:symm}.

\section{Formulation of the problem and main results}\label{sec:form}
Consider a rigid body, $\mathscr B$, occupying the closure of the bounded domain $\Omega_0\subset\mathbb{R}^3$, completely immersed in a Navier-Stokes liquid, $\mathscr L$, that fills the entire three-dimensional space outside $\mathscr B$, see Figure 1. 
\begin{figure}[ht!]
{\includegraphics[width=3.75in,height=2.8in,keepaspectratio]{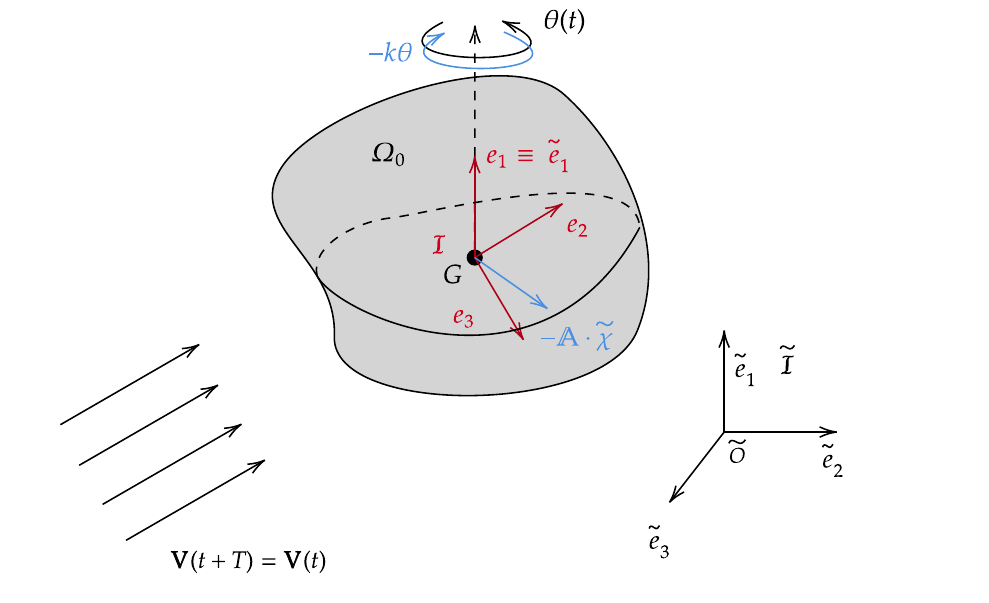}}
{\includegraphics[width=3.75in,height=2.8in,keepaspectratio]{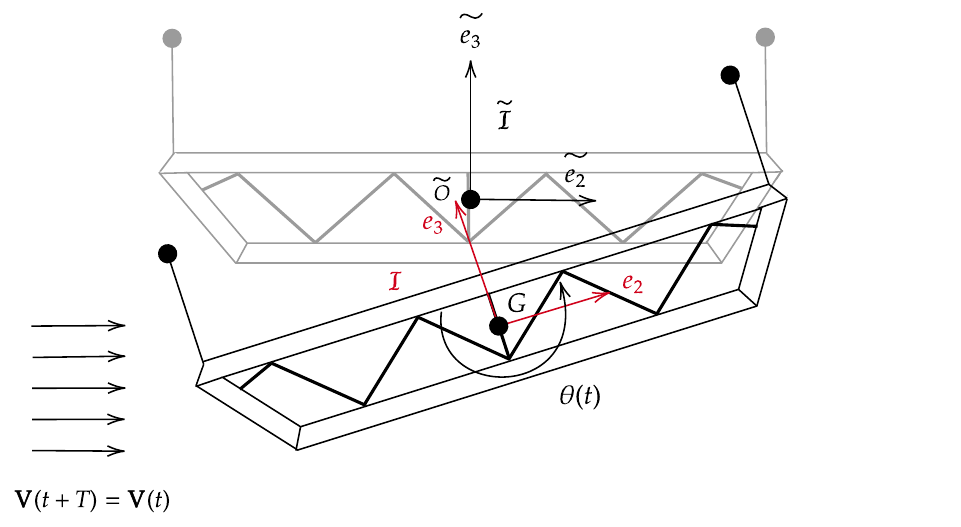}}
\caption{\small  Left: a generic representation of the system body-liquid. Right: a 2D-schematic view of the cross-section of the deck of a bridge subjected to a periodic wind-flow.} 
\end{figure}
We suppose that, with respect to an inertial frame $\tilde\cali:=\{\tilde O,\tilde\bfe_1,\tilde\bfe_2,\tilde\bfe_3\}$, $\mathscr B$  moves subject to the following force and torque:\begin{itemize} \item[(a)] A linear, possibly anisotropic, restoring force $\bfR=-{\mathbb A}\cdot{\tilde\bfchi}$ with ${\mathbb A}$ symmetric, positive definite matrix ({\em stiffness matrix}) and ${\tilde\bfchi}:=\tilde\bfy_G-\tilde\bfy_{\tilde O}$ displacement of the center of mass $G$ of $\mathscr B$ with respect to  $\tilde O$; 
\item[(b)] A restoring torque $-{k}\theta\tilde\bfe_1$ with ${k}>0$ (\textit{stiffness constant}) and $\theta$ angle counted with respect to the direction $\tilde\bfe_2$. 
\end{itemize}
 We will assume that $\mathscr B$ is free to rotate only around the (fixed) direction  $\tilde{\bfe}_1$. 
\smallskip\par 
We further suppose that the motion of the coupled system $\mathscr S:=\mathscr B-\mathscr L$ is driven by a uniform, time-periodic flow at ``large" distances from $\mathscr B$, characterized by the velocity field $-V(t)\tilde{\bfb}_\alpha$, where $V$ is a prescribed time-periodic  real function of period $T>0$ and  $\tilde{\bfb}_\alpha=\cos\alpha\, \tilde\bfe_1 +\sin\alpha\, \tilde\bfb$, where $\tilde\bfb = (0,\tilde b_2,\tilde b_3)$ has length $1$. The (given) angle $\alpha$ can be seen as an angle of attack. Thus, when $\alpha=0$, the far-field flow is aligned with the axis of rotation of the body, whereas at $90^\circ$, it  is totally transverse. Under these conditions, 
the equations governing the motion of $\mathscr S$ in the frame $\tilde\cali$ can then be written, in dimensionless form, as follows (see, e.g., \cite[Section 1]{Gah}) 
\be\left\{\ba{l}\medskip\left.\ba{r}\medskip
\partial_t\tilde{\bfw}+\lambda\tilde{\bfw}\cdot\nabla\tilde{\bfw}=\Delta\tilde{\bfw}-\nabla {\tilde{q}}\\
\Div\tilde{\bfw}=0\ea\right\}\ \ \mbox{in $\BCup{t\in\real}{}\left(\tilde\cald_{\tilde y}(t)\times\{t\}\right)$}\,,\\ \medskip
\ \ \tilde{\bfw}(\tilde\bfy,t)={\tilde\bfeta}(t)+\omega(t)\tilde\bfe_1\times(\tilde{\bfy}-\tilde{\bfy}_G(t))\,, \ \mbox{at $\BCup{t\in\real}{}\left(\partial\tilde\cald_{\tilde y}(t)\times\{t\}\right)$}\,;\ \\ \medskip
\ \ \Lim{|\tilde{\bfy}|\to\infty}\tilde{\bfw}(\tilde\bfy,t)=-\bfV(t):=- V(t)\tilde\bfb_\alpha\,,\ t\in\real\,,\\
\medskip\left.\ba{r}\medskip
\dot{{\tilde\bfeta}}+{\mathbb A}\cdot{\tilde\bfchi}+\varpi\Int{\partial\tilde\cald_{\tilde y}(t)}{} \mathbb{S}(\tilde{\bfw},{\tilde q})\cdot\tilde\bfn=\0
\\ \medskip
\dot{{\tilde\bfchi}}={{\tilde\bfeta}}\\ \medskip
\dot{\omega}+k\theta+\tau\tilde\bfe_1\cdot\Int{\partial\tilde\cald_{\tilde y}(t)}{}(\tilde{\bfy}-\tilde\bfy_G)\times \mathbb{S}(\tilde{\bfw},{\tilde q})\cdot\tilde\bfn=0\, \\
\medskip
\dot{\theta}={\omega}
\ea\right\}\ \ \mbox{in $\real$\,.}
\ea\right.
\eeq{01}
Here $\tilde\cald_{\tilde y}(t)=\tilde\cald_{\tilde y}({\tilde\bfchi}(t),\theta(t))$, the domain occupied by $\mathscr L$ at time $t$, depends on the position of the center of mass of the body and its orientation, i.e. 
$$\cald_{\tilde y}(t)= \left(\tilde\bfy_{\tilde O} + {\tilde\bfchi}(t) + \mathbb{Q}(\theta(t))\Omega_0\right)^\complement,$$
where $\mathbb Q(t)=\mathbb{Q}(\theta(t))$,  $t\in \mathbb{R}$, is the one-parameter
family of rotations around  $\tilde\bfe_1$ defined by 
\be
\mathbb{Q}(\theta(t)):=\left(\ba{ccc}\medskip 1 &0&0\\ \medskip
0&\cos\theta(t)&-\sin\theta(t)\\
0&\sin\theta(t)&\cos\theta(t)\ea\right)\,.
\eeq{Q}   
Moreover,
 $\tilde{\bfw}$ and ${\tilde q}$ are velocity and pressure fields of the liquid, and
$$
 \mathbb{S}(\bfz,\psi):=2\,\mathbb D(\bfz)-\psi\,\mathbb I\,,\ \ \ \mathbb D(\bfz):=\frac12\big(\nabla\bfz+(\nabla\bfz)^\top\big)\,,
$$
with $\mathbb I$ identity matrix, is the (dimensionless) Cauchy stress tensor, and  $\tilde\bfn$ the unit outer normal at $\partial\tilde\cald$ directed toward $\mathscr{B}$.
Notice that, with the above  non-dimensionalization, we have 
\be
\sup_{t\in \real}|V(t)|=1\,.
\eeq{SV}
Finally,  
${\lambda},\mathbb A$, $k$, $\varpi$ and $\tau$ are non-dimensional physical quantities (respectively the relevant Reynolds number, the non-dimensional linear and angular stiffness parameters, and the non-dimensional ratios of densities). 
We will count the angle $\theta$ \emph{with its multiplicity} so that, in particular, the restoring torque takes into account the number of full turns made by the body. 

Our ultimate goal is to show that, for any given $\lambda, \mathbb A, k,\varpi, \tau$ and $T>0$, and any  (sufficiently smooth) $V$, problem \eqref{01} has at least one,  suitably defined,  time-periodic of period $T$ (that we will refer to as $T$-periodic) weak solution $(\tilde{\bfw},{\tilde\bfchi},\theta)$. We will achieve this goal in its full generality, namely, without restriction on the size of the data, and, even more importantly, for \emph{arbitrary} choice of the period $T$. Furthermore, we shall prove, in particular, a \emph{pointwise control of the amplitude} of the oscillations of $\mathscr B$ in terms of the data. Thus, a significant consequence of our result is that the viscosity of the liquid, {\em no matter how small},  is able to damp any possible source of resonance which, therefore, cannot take place within this model. 
\par As a last remark, we observe that our method of proof drastically simplifies in the particular case where $\mathscr B_0$ is invariant with respect to the group of rotations around the axis $\tilde\bfe_1$. This simplified approach will be presented in Section \ref{sec:symm}.

\subsection{Body-fixed frame formulation}\label{sec:body-fixed}

Our next objective is to reformulate problem \eqref{01} into an equivalent problem by a  two-step procedure. To simplify the notations, if  $\bfa\equiv(a_1,a_2,a_3)\in\mathbb{R}^3$, we set $\bfa^\perp:=\tilde\bfe_1\times \bfa=(0,-a_3,a_2)$.  It is important to highlight that the usual method, used for instance in \cite{BoGa}, to formulate the problem with a velocity field vanishing at infinity by constructing a suitable lifting $\bfU$ of $\bfV$ is not appropriate to reformulate \eqref{01} due to the angular degree of freedom. Indeed, this method generally involves searching for a velocity field $\tilde{\bfv}$ of the form 
$$
\tilde{\bfv}:=\tilde{\bfw}-V\tilde{\bfb}_\alpha+\bfU\,,
$$
where $\bfU$ is a compactly supported function equal to $V\tilde{\bfb}_\alpha$ near the obstacle $\Omega_0$, and $\tilde{\bfw}$ is the new unknown velocity field. However, when writing the problem for such a $\tilde{\bfw}$ in a body-fixed frame,  the fluid equation contain  terms that couple the angular velocity $\dot{\theta}$ and the far-field velocity $V\tilde{\bfb}_\alpha$, which do not seem tractable to us. 
 
As an alternative, we rather attach first the frame origin at the center of  mass $G$ of $\mathscr B$, scale the velocity field $\tilde{\bfw}$ with the uniform flow $V(t)\tilde\bfb_\alpha$, and we redefine the pressure field accordingly. Precisely, we set
	\be
		\tilde{\bfx}:=\tilde{\bfy}-\tilde\bfy_G(t)\,, \quad \tilde{\bfu}:=\tilde{\bfw}- V(t)\tilde\bfb_\alpha\,,\quad \tilde{p}:={\tilde q}-\lambda\dot{{ V}}(t)\tilde\bfb_\alpha\cdot\tilde{\bfx}\,. 
	\eeq{change}
In this system of coordinates, the fluid domain $\tilde{\cald}_{\tilde x}(t)=\tilde{\cald}_{\tilde x}(\theta(t))$ does only depend on the orientation of the solid, i.e.
$$\tilde{\cald}_{\tilde x}(t)= \left(\mathbb{Q}(\theta(t))\Omega_0\right)^\complement.$$
This yields the new couple system in the frame $\tilde\cali_G:=\{G,\tilde\bfe_1,\tilde\bfe_2,\tilde\bfe_3\}$
\be\left\{\ba{l}\medskip\left.\ba{r}\medskip
\partial_t\tilde{\bfu}+\lambda(\tilde{\bfu}-\tilde\bfgamma(t))\cdot\nabla\tilde{\bfu}=\Div \mathbb{S}(\tilde{\bfu},\tilde{p})\\
\Div\tilde{\bfu}=0\ea\right\}\ \ \mbox{in $\BCup{t\in\real}{}\left(\tilde{\cald}_{\tilde x}(t)\times\{t\}\right)$}\,,\\ \medskip
\ \ \tilde{\bfu}(\tilde{\bfx},t)=\tilde\bfgamma(t)+\omega(t)\tilde{\bfx}^\perp\,, \ \mbox{at $\BCup{t\in\real}{}\left(\partial\tilde{\cald}_{\tilde x}(t)\times\{t\}\right)$}\,;\ \\ \medskip
\ \ \Lim{|\tilde{\bfx}|\to\infty}\tilde{\bfu}(\tilde{\bfx},t)={\bf{0}}\,, \ \ t\in\real\,,\\
\medskip\left.\ba{r}\medskip
\dot{{\tilde\bfgamma}}+{\mathbb A}\cdot{{\tilde\bfchi}}+\varpi\Int{\partial\tilde{\cald}_{\tilde x}(t)}{} \mathbb{S}(\tilde{\bfu},\tilde{p})\cdot\tilde\bfn=c\,\dot{V}(t)\tilde\bfb_\alpha
\\ \medskip
\dot{{\tilde\bfchi}}={\tilde\bfgamma-V(t)\tilde\bfb_\alpha}\\ \medskip
\dot{\omega}+k\theta+\tau\,\bfe_1\cdot\Int{\partial\tilde{\cald}_{\tilde x}(t)}{}\tilde{\bfx}\times \mathbb{S}(\tilde{\bfu},\tilde{ p})\cdot\tilde\bfn=d\,\dot{V}(t)\, \\
\medskip
\dot{\theta}={\omega}
\ea\right\}\ \ \mbox{in $\real$\,,}
\ea\right.
\eeq{01+++only-tranlsation}
where 
\be\ba{ll}
\medskip
c=c(\varpi,\mathscr B):=(1-{ \lambda}\varpi\,{\rm vol}\,(\mathscr B)),\ 
 d=d(\alpha,\tau,\mathscr B):=\lambda\tau\sin \alpha\,  \tilde\bfb\cdot\int_{\Omega_0} (-x_3\bfe_2+x_2\bfe_3 )   \,. 
\ea
\eeq{1.6666}
Second, as customary in this type of investigations, we rewrite  \eqref{01+++only-tranlsation} in a body-fixed frame $\cali$, so that the domain occupied by the liquid becomes time-independent.  
To this end, we choose $\cali:=\{G,\bfe_1,\bfe_2,\bfe_3\}$ where 
$\bfe_i=\mathbb Q^\top(t)\cdot\tilde\bfe_i$, $i=1,2,3$, so that we use the notation 
$$\bfa(t):=\mathbb Q^\top(t)\cdot \tilde{\bfa}(t)$$
for vectors.
We then also introduce new variables
\be
\bfx:=\mathbb Q^\top(t)\cdot \tilde{\bfx} \,,
\eeq{change_rotation}
that yield the steady fluid domain
$$
\Omega:=\left\{\bfx\in\real^3:\, \bfx=\mathbb Q^\top(t)\cdot\tilde{\bfx}\,,\ \tilde{\bfx}\in  \tilde\cald_{\tilde x}(t), \, t\in\real\right\} =\mathbb R^3\setminus\Omega_0\,.
$$
This change of variables induces the transformed velocity and presure fields
\be
\ba{l}\medskip
(\bfu(\bfx,t),p(\bfx,t)):=(\mathbb Q^\top(t)\cdot\tilde{\bfu}(\mathbb Q(t)\cdot\bfx,t),\tilde{p}(\mathbb Q(t)\cdot\bfx,t))\,,\\ 
\ea
\eeq{change_var}
whereas the associated Cauchy stress tensor now reads
$$\mathbb{T}(\bfu,p):=\mathbb{Q}^\top(t)\cdot \mathbb{S}(\mathbb{Q}(t)\cdot \bfu,p)\cdot\mathbb{Q}(t).$$
Then, see for instance \cite[\S\S 1 and 2.1]{Gah} for the detailed computation, $(\bfu,p,\bfxi,\bfdelta,\omega,\theta)$ solve the following set of equations \smallskip\par\noindent
\be\left\{\ba{l}\medskip\left.\ba{r}\medskip
\partial_t\bfu+\lambda(\bfu-\bsfV)\cdot\nabla\bfu+\dot{\theta}\bfu^\perp= \Div\mathbb T(\bfu,p)\\
\Div\bfu=0\ea\right\}\ \ \mbox{in $\Omega\times\real$}\,,\\ \medskip
\ \ \bfu({\bfx},t)=\bsfV({\bfx},t) \mbox{ on $\partial\Omega_0\times\real$}\,;\ \\ \medskip
\ \ \Lim{|\bfx|\to\infty}\bfu(\bfx,t)={\bf0}\,,\ t\in\real\,,
\\ \hspace{2mm}\bfu(\bfx,0)=\bfu_0\ \mbox{in $\Omega$}
\\
\medskip
\left.\ba{r}\medskip
\dot{\bfxi}+\dot\theta\bfxi^\perp+\mathbb B(t)\cdot\bfdelta+\varpi\Int{\partial\Omega}{}\mathbb T(\bfu,p)\cdot\bfn= c\, \dot V(t) \bfb_\alpha(t)
\\ \medskip
\dot{\bfdelta}+\dot\theta\bfdelta^\perp={\bfxi-V(t)\bfb_\alpha(t)}\\ \medskip
\dot{\omega}+k\theta+\tau\,\bfe_1\cdot\Int{\partial\Omega}{}\bfx\times\mathbb T(\bfu,{ p})\cdot\bfn=d\, \dot V(t)\, \\
\medskip
\dot{\theta}={\omega}
\ea\right\}\ \ \mbox{in $\real$\,,}
\ea\right.
\eeq{01-translation+rotation_or}
where 
\be\bsfV({\bfx},t):=\bfxi(t)+\omega(t){\bfx}^\perp
\eeq{boldV}
and 
\be
\mathbb B(t)=\mathbb{B}(\theta(t)):=\mathbb Q^\top(\theta(t))\cdot\mathbb A\cdot\mathbb Q(\theta(t))\,.
\eeq{A}
Notice that, for any $t\in\real$, this matrix is  positive definite. In particular, there exist $\rho_1,{\rho_2}>0$ depending only on the eigenvalues of $\mathbb{A}$ such that 
\be
\rho_1 |\bfdelta|^2\le \bfdelta\cdot \mathbb{B}\cdot \bfdelta\le {\rho_2}\,|\bfdelta|^2\,, \ \ \forall\, \bfdelta\in\real^3.
\eeq{B_bound}

\begin{remark}
For future reference, we emphasize that, in view of their definition in \eqref{A}, even though the matrix $\mathbb{B}$ and the vector ${\bfb}_\alpha$ are now both time dependent and their direction is an unknown of the system, they are known functions of $\theta$.
\end{remark}

\subsection{$T$-periodic weak solutions and statement of the main result}
\label{Sec:main}
We begin by recalling  some standard notations. For any open set $\Omega\subseteq\mathbb{R}^3$, $q\in[1,\infty]$ and $m\in\mathbb{N}$, $L^q(\Omega)$ and $W^{m,2}(\Omega)$ denote respectively the usual Lebesgue and Sobolev spaces. For any Banach space $X$, we use $\|\cdot\|_X$ to indicate its norm when a confusion is possible, while we denote by $L^q(0,T;X), W^{1,q}(0,T;X)$ and $C^m(0,T;X)$ the standard Bochner spaces of functions defined from $(0,T)$ to $X$. We use $H^1(0,T;X)$ to denote $W^{1,2}(0,T;X)$.
 \par
In order to give the definition of $T$-periodic weak solution to \eqref{01-translation+rotation_or}, we need to introduce some further function spaces commonly used in the context of fluid-structure interaction. Specifically, we set
 $$
 \begin{aligned}
 	\mathcal{K}=\mathcal{K}(\mathbb{R}^3)&:=\{\bfphi\in C^\infty_0(\mathbb{R}^3)\,:\, \exists \,\hat{\bfrho}\in\mathbb{R}^3, \hat{\alpha}\in\mathbb{R}\,\,\text{s.t.}\,\, \bfphi(\bfx)=\hat{\bfrho}+\hat{\alpha}\bfx^\perp \,\,\text{for $\bfx$ in a neighborhood of $\Omega_0$}\}\,,
 	\\ \mathcal{C}=\mathcal{C}(\mathbb{R}^3)&:=\{\bfphi\in\mathcal{K}(\mathbb{R}^3)\,:\,\text{div}\bfphi=0\,\,\,\text{in}\,\,\,\mathbb{R}^3\}\,,
 	 \end{aligned}
 	$$
 	and define in $\mathcal{K}$ the following scalar product
 \be
 \langle\bfphi_1,\bfphi_2\rangle_{\mathbb{R}^3}:=(\bfphi_1,\bfphi_2)_{\Omega}+{\frac1\varpi}\hat{\bfrho}_1\cdot \hat{\bfrho}_2+{\frac1\tau}\hat{\alpha}_1\hat{\alpha}_2\,,	\qquad \bfphi_1,\bfphi_2\in \mathcal{K}\,,
 \eeq{sp}
where $(\cdot,\cdot)_{A}$, for any open $A\subseteq \mathbb{R}^3$,  is the usual $L^2(A)$-scalar product (we may omit the subscript $``A"$ when no confusion arises).  We also define the spaces
 $$
 \begin{aligned}
 		\mathcal{L}^2(\mathbb{R}^3)&:=\{\bfphi\in L^2(\mathbb{R}^3)\,:\,\bfphi|_{\Omega_0}=\hat{\bfrho}+\hat{\alpha}\bfx^\perp \text{ for some $\hat{\bfrho}\in\mathbb{R}^3,\hat{\alpha}\in\mathbb{R}$}\}\,,
 \\\mathcal{H}(\mathbb{R}^3)&:=\{\text{completion of $\mathcal{K}(\mathbb{R}^3)$ in the norm induced by \eqref{sp}}\}\,,
 	\\\mathcal{D}^{1,2}=\mathcal{D}^{1,2}(\mathbb{R}^3)&:=\{\text{completion of $\mathcal{C}(\mathbb{R}^3)$ in the norm $\|\mathbb{D}(\cdot)\|_2$} \}\,.
 \end{aligned}
 $$
 Equipped with the scalar product 
 $
 (\mathbb{D}(\cdot),\mathbb{D}(\cdot))
$,
the space $\mathcal{D}^{1,2}$ is an Hilbert space.
\medbreak
We also need to define the class $\mathcal{C}_{ \sharp}(\mathbb{R}^3)$ of test functions constituted by $T$-periodic $\bfphi\in C^1(\mathbb{R}^3 \times \mathbb{R}\,;\mathbb{R}^3)$ restricted to $[0,T]$ and satisfying:
\begin{enumerate}[(a)] 
	\item $\text{div}\,\bfphi(\bfx,t)=0$ for $(\bfx,t)\in \mathbb{R}^3\times\mathbb{R}$ ;
	\item $\bfphi(\bfx,t)=\hat{\bfrho}(t)+\hat{\alpha}(t)\bfx^\perp$ for some $\hat{\bfrho}\in C^1(\mathbb{R};\mathbb{R}^3)$, $\hat{\alpha}\in C^1(\mathbb{R};\mathbb{R})$ and $\bfx$ in a neighborhood of $\Omega_0$ and $t\in \mathbb{R}$ .
	\item supp$_{\bfx} \bfphi(\bfx,t)\subset \mathbb{R}^3$ for all $t\in\mathbb{R}$ .
	\item $\bfphi(\bfx,t+T)=\bfphi(\bfx,t)$ for all $(\bfx,t)\in \mathbb{R}^3\times \mathbb{R}$ .
\end{enumerate}
Thus, formally multiplying \eqref{01-translation+rotation_or}$_1$ by $\bfphi\in\mathcal{C}_{ \sharp}(\mathbb{R}^3)$, integrating by parts over $\Omega\times[0,T]$ and using \eqref{01-translation+rotation_or}$_{2-8}$, we obtain
$$
\begin{aligned}
\langle\bfu(T),\bfphi(T)\rangle-\langle\bfu(0),\bfphi(0)\rangle=&-\int_0^T\bigg[-\langle\bfu,\bfphi_t\rangle+\lambda((\bfu-\bsfV)\cdot\nabla\bfu+\dot{\theta}\bfu^{\perp},\bfphi)+{\frac1\varpi}\dot{\theta}\bfxi^{\perp}\cdot\hat{\bfrho}\\&
+2(\mathbb{D}(\bfu),\mathbb{D}(\bfphi))
+{\frac1\varpi}\hat{\bfrho}\cdot\mathbb{B}\cdot\bfdelta+{\frac{\hat{\alpha}k\theta}\tau}-{\frac{c\dot{V}}\varpi}\,\hat{\bfrho}\cdot {\bfb}_\alpha-{\frac{\hat{\alpha}d\dot{{ V}}}\tau}\bigg]\,{\rm d}t\,.
\end{aligned}
$$
By \eqref{01-translation+rotation_or}$_7$-\eqref{01-translation+rotation_or}$_9$, we also deduce 
$$
\bfdelta(T)-\bfdelta(0)=\int_0^T(\bfxi(t)-V(t){\bfb}_\alpha(t)-\dot{\theta}(t)\bfdelta^\perp(t))\,{\rm d}t$$ 
and
$$\theta(T)-\theta(0)=\int^T_0\omega(t)\,{\rm d}t\,.
$$
Thus, if $(\bfu,\bfxi,\bfdelta,\omega,\theta)$ is a (sufficiently smooth) $T$-periodic solution to \eqref{01-translation+rotation_or}, it satisfies 
\be\ba{rl}\smallskip
&\Int0T\bigg[-\langle\bfu,\bfphi_t\rangle+\lambda((\bfu-\bsfV)\cdot\nabla\bfu+\dot{\theta}\bfu^{\perp},\bfphi)+{\frac{\dot{\theta}}\varpi}\bfxi^{\perp}\cdot\hat{\bfrho}\\\smallskip&
	\quad\qquad +2(\mathbb{D}(\bfu),\mathbb{D}(\bfphi))
	+{\frac1\varpi}\hat{\bfrho}\cdot\mathbb{B}\cdot\bfdelta+{\frac{\hat{\alpha}k\theta}\tau}-{\frac{c\dot{V}}\varpi}\,\hat{\bfrho}\cdot {\bfb}_\alpha-{\frac{\hat{\alpha}d\dot{{ V}}}\tau}\bigg]\,{\rm d}t=0\,,\\ 
&\dot{\bfdelta}+\dot\theta\bfdelta^\perp={\bfxi-V(t)\bfb_\alpha(t)}\,,\quad\Int0T(\bfxi(t)-V(t){\bfb}_\alpha(t)-\dot{\theta}(t)\bfdelta^\perp(t))\,{\rm d}t =0\,,\\ 
&\dot{\theta}=\omega\,,\quad \Int0T\omega(t)\,{\rm d}t=0\,,
\end{array}
\eeq{periodic_weak}
for any test function $\bfphi\in\mathcal{C}_{\sharp}(\mathbb{R}^3)$.
This computation motivates the definition of weak solution.
\begin{definition}\label{definition_weak}
	We say that $(\bfu,\bfxi,\bfdelta,\omega,\theta)$ is a {\textit{$T$-periodic weak solution}} to \eqref{01-translation+rotation_or} if 
	\begin{enumerate}[(i)] 
	\item $\bfxi\in L^2(0,T;\mathbb{R}^3)$, $\bfdelta\in H^1(0,T;\mathbb{R}^3)$, $\omega\in L^{2}(0,T;\mathbb{R})$, $\theta\in H^1(0,T;\mathbb{R})$\, ;
	\item 
$\bfu\in L^2(0,T;\mathcal{D}^{1,2}(\mathbb{R}^3))$, with $\bfu(\bfx,t)|_{\partial \Omega}=\bfxi(t)+{\omega}(t)\bfx^\perp$ for almost every $t\in [0,T]$\,;
		\item $(\bfu,\bfxi,\bfdelta,\omega,\theta)$ satisfies \eqref{periodic_weak} for any test function $\bfphi\in\mathcal{C}_{\sharp}(\mathbb{R}^3)$\,.
	\end{enumerate}
\end{definition}
\par
We are now in a position to state our main result.  
\begin{theorem}\label{main:th}
	Let $V\in H^1(0,T;\mathbb{R})$ be $T$-periodic for some $T>0$. Then, there is at least one corresponding $T$-periodic weak solution to \eqref{01-translation+rotation_or}. Furthermore, there exists $C>0$ depending only on the physical and geometrical parameters of the problem, such that
	\be
\begin{array}{ll}\smallskip
\|\bfxi\|_{L^2}^2+\|\omega\|_{L^2}^2 + \Int0T\|\nabla \bfu(t)\|_{2}^2\,{\rm d}t\le C\,\mathcal{V}\,,
\\ \|\bfdelta\|_{L^2}^2+\|\theta\|_{L^2}^2 \le C\mathcal{V}\,\left(T^2+{\mathcal{V}}+\Frac{\mathcal{V}^2}{T^{1/2}}+\Frac{\mathcal{V}^3}T\right)\,,
\end{array}
\eeq{bound_weaksol}
and 
\be
\|\bfdelta\|_{L^\infty}+\|\theta\|_{L^\infty} \le C\mathcal{V}^{1/2} \left(T^{1/2}+T\mathcal{V}^{1/2}+\mathcal{V}+\frac{\mathcal{V}^{3/2}}{T^{1/4}}+\frac{\mathcal{V}^{1/2}+\mathcal{V}^{2}}{T^{1/2}}+\frac{\mathcal{V}}{T^{3/4}}+\frac{\mathcal{V}^{3/2}}T\right)\,,
\eeq{bound_point}
where
$$
\mathcal{V}:=\|V\|_{H^1}^2=\int^{T}_0\left(|V(t)|^2+|\dot{V}(t)|^2\right)\,{\rm d}t\,.
$$
\end{theorem}
\Br{\em 
It is worth emphasizing that the theorem does not impose any restriction on the period $T$.  In addition,  
the estimate in \eqref{bound_point} shows that the amplitude of linear and angular oscillations is pointwise controlled by the data and the period $T$.  These two properties combined therefore exclude the occurrence of any resonance phenomenon.}
\ER{th_1}
\par
The proof of Theorem \ref{main:th} is developed as follows.  
We begin to prove global-in-time well-posedness of the initial-boundary value problem for a suitable mollification \eqref{01-translation+rotation_or}--\eqref{boldV} in a sequence of bounded domains, $(\Omega_k)_k$ whose union coincides with $\Omega$; see \eqref{01-translation+rotation} and Lemma~\ref{Lemma_glob_ex}. For each $k$, we then construct  the  Poincar\'e map ${\sf P}$ bringing initial data to the corresponding solutions at time $T$ in the ``total energy" space, and look for a fixed point for {\sf P}.  To deduce the latter, the crucial difficulty comes from the fact that there is no damping mechanism in the structure (both the restoring force and torque are {\em undamped}) so  that the only dissipation comes from the interaction of the body with the viscous fluid. However, as is well known \cite{PaMa}, the {\em exponential} decay of the {\em total} energy of the coupled system is a critical tool to establish that ${\sf P}$ has a fixed point. This is equivalent to demonstrating that the above interaction produces the dissipation not only of the total kinetic energy, but also of the   potential energy, i.e.  displacement and rotation of the body.  To show this property, we use a generalization of the methods introduced in \cite{BoGa}.  Due to the more complicated model studied here, we  prove  that the above type of decay happens at least locally, namely, over the time-interval $[0,T]$; see Lemma~\ref{Lemma_dissipation}.  However, this  
suffices to prove the existence of a fixed point for {\sf P}, and hence the existence of a $T$-periodic solution, ${\sf p}_k$ (say), in every $\Omega_k$; see Proposition~\ref{prop:weak_per_bounded}. The next step is to let $k\to\infty$ and show that $({\sf p}_k)_k$ converges, on a subsequence at least, to a $T$-periodic weak solution to the original problem. In order to accomplish this last step, it is essential to ensure that $({\sf p}_k)_k$ is uniformly bounded in $k$ in the function space to which weak solutions belong. While this property for the total {\em kinetic} energy (see \eqref{bound_Rindependent}$_1$) can be deduced along the same lines as \cite{BoGa}, the uniform,  pointwise boundedness of displacement and rotation (see \eqref{bound_Rindependent}$_2$, and \eqref{bound_point_two}) requires fresh ideas. To deduce the latter, we use, in the definition of weak solution \eqref{periodic_weak_bounded}$_1$, a special test function whose trace at the boundary coincides with the average over a period of the displacement and rotation fields. This allows us to obtain an estimate for these quantities that, once combined with the Poincar\'e-Wirtinger inequality, the estimate on the kinetic energy and embedding theorems, furnishes the desired result; see the proof of Proposition~\ref{prop:weak_per_bounded}. With these results in hand, we then pass to the limit $k\to\infty$, possibly along a subsequence of $({\sf p}_k)_k$, and show that the limiting velocity field is a $T$-periodic weak solution to the original problem which, in addition, satisfies the same uniform bounds in terms of the data as {\sf p}$_k$, thus completing the proof of Theorem~\ref{main:th}; see Section~\ref{Sec8}.

\section{The initial-boundary value problem in bounded domains}\label{sec:IVBP}
The main objective of this section is to prove well-posedness together with uniform energy estimates for the following regularized version of \eqref{01-translation+rotation_or}
\be\left\{\ba{l}\medskip\left.\ba{r}\medskip
\partial_t\bfu+\lambda(k_\eta * \bfu-\bsfV)\cdot\nabla\bfu+\dot{\theta}\bfu^\perp= \Div\mathbb T(\bfu,p)\\
\Div\bfu=0\ea\right\}\ \ \mbox{in $\Omega_R\times(0,\infty)$}\,,\\ \medskip
\ \ \bfu({\bfx},t)=\bsfV({\bfx},t):=\bfxi(t)+\omega(t){\bfx}^\perp\, \ \mbox{at $\partial\Omega_0\times(0,\infty)$}\,;\ \\ \medskip
\ \ \bfu(\bfx,t)={\bf0}\,\ \mbox{at $\partial B_R\times (0,\infty)$}\,; \\ \medskip
\left.\ba{r}\medskip
\dot{\bfxi}+\dot\theta\bfxi^\perp+\mathbb B(t)\cdot\bfdelta+\varpi\Int{\partial\Omega}{}\mathbb T(\bfu,p)\cdot\bfn= c\, \dot V(t) \bfb_\alpha(t)
\\ \medskip
\dot{\bfdelta}+\dot\theta\bfdelta^\perp={\bfxi-V(t)\bfb_\alpha(t)}\\ \medskip
\dot{\omega}+k\theta+\tau\,\bfe_1\cdot\Int{\partial\Omega}{}\bfx\times\mathbb T(\bfu,{ p})\cdot\bfn=d\, \dot V(t)\, \\
\medskip
\dot{\theta}={\omega}
\ea\right\}\ \ \mbox{in $(0,\infty)$,}
\ea\right.
\eeq{01-translation+rotation}
where $(k_\eta)_{\eta>0}$ is a family of mollifiers. We will consider the boundary value problem \eqref{01-translation+rotation} with compatible initial conditions $(\bfu_0,\bfxi_0,\bfdelta_0,\theta_0,\omega_0)$, that is $\bfu_0|_{\Omega_0}=\bfxi_0+\omega_0\bfx^\perp$. We thus add the conditions 
\be\left\{\ba{l}\medskip
\bfu(\bfx,0)=\bfu_0\, \ \mbox{in $\Omega_R$}\,;
\\
\medskip
\bfxi(0)=\bfxi_0\,,\quad\bfdelta(0)=\bfdelta_0\,,\quad \theta(0)=\theta_0\,,\quad \omega(0)=\omega_0\,.
\ea\right.
\eeq{01-translation+initial-cond}

In comparison to the original coupled system \eqref{01-translation+rotation_or}, the fluid domain has been restricted to the bounded open set
\be
\Omega_{R}:=\Omega\cap B_R, 
\eeq{omegaR}
an a homogeneous Dirichlet condition has been added on $\partial B_R$ (for all time) and the convective term $\bfu\cdot \nabla \bfu$ in the Navier-Stokes equations has been replaced by the smoother expression $\bfu_\eta\cdot \nabla \bfu$ where
\be
\bfu_\eta =k_\eta * \bfu
\eeq{u_eta}
is the (Friederichs) mollification of the velocity field $\bfu$. As usual, the family of kernels $(k_\eta)_{\eta>0}$ satisfies
$$
k_\eta(\bfxi):=\eta^{-3}k(\bfxi/\eta)\,; \quad k\in C_0^\infty(B_1)\,,\quad \int_{\mathbb{R}^3}k(\bfx)=1\,. 
$$
To be completely precise, we consider the system  \eqref{01-translation+rotation} for $R>\text{diam}(\Omega_0)$ and $\eta<\eta_0$ only, where 
$\eta_0>0$ is chosen small enough in order to ensure that the set
$$
\Omega_{00}:=\{\bfx\in\Omega_0\,:\,\, \text{dist}(\bfx,\partial\Omega)>\eta_0\}
$$
is not empty. We will often use the well known estimate
\be
\|\bfu_\eta\|_\infty\le c_\eta\|\bfu\|_2\,,
\eeq{moly}
which holds for all $\eta\in (0,\eta_0)$ and some $c_\eta>0$. 

\medbreak

To study the system \eqref{01-translation+rotation}, we need  versions of the spaces $\mathcal{L}^2$, $\mathcal{H}$, and $\mathcal{D}^{1,2}$ defined on $B_R$, that is 
 $$
 \begin{aligned}
 	\mathcal{L}^2(B_R)&:=\{\bfphi\in L^2(B_R)\,:\,\bfphi|_{\Omega_0}=\hat{\bfrho}+\hat{\alpha}\bfx^\perp \text{ for some $\hat{\bfrho}\in\mathbb{R}^3,\hat{\alpha}\in\mathbb{R}$}\}\, ;
 \\\mathcal{H}(B_R)&:=\{\bfphi\in\mathcal{L}^2(B_R)\,:\,\text{div}\bfphi=0\,\,,\,\bfphi\cdot\bfn|_{\partial B_R}=0\}\, ;
 \\\mathcal{D}^{1,2}(B_R)&:=\{\bfphi\in H^1(B_R)\,:\,\text{div}\,\bfphi=0\,,\bfphi|_{\Omega_0}=\hat{\bfrho}+\hat{\alpha}\bfx^\perp \text{ for some $\hat{\bfphi}\in\mathbb{R}^3,\,\hat{\alpha}\in\mathbb{R}$}, \,\bfphi|_{\partial B_R}={\bf 0}\}\,.
 \end{aligned}
 $$  
 The spaces $\mathcal{H}(B_R)$ and $\mathcal{D}^{1,2}(B_R)$ are Hilbert spaces when endowed with scalar products 
 $$
 \begin{aligned}
&  \langle\bfphi_1,\bfphi_2\rangle_{B_{R}}:=(\bfphi_1,\bfphi_2)_{\Omega_R}+{\frac1\varpi}\hat{\bfrho}_1\cdot \hat{\bfrho}_2+{\frac1\tau}\hat{\alpha}_1\hat{\alpha}_2\ \,,\ \bfphi_i\in\mathcal{H}(B_R)\,;\\& (\mathbb{D}(\bfpsi_1),\mathbb{D}(\bfpsi_2))_{B_R}\,,\ \bfpsi_i\in \mathcal{D}^{1,2}(B_R)\,,\quad i=1,2\,.
 \end{aligned}
 $$
We next motivate the definition of weak solutions to  \eqref{01-translation+rotation}-\eqref{01-translation+initial-cond}. We thus multiply \eqref{01-translation+rotation}$_1$ by $\bfpsi\in\mathcal{D}^{1,2}(B_R)$, formally integrate by parts over $\Omega_R\times (0,\infty)$ and  take into account \eqref{01-translation+rotation}$_{2-8}$ to deduce that
\be
\begin{aligned}
\langle\bfu(t)-\bfu_0,\bfpsi\rangle=-\int_0^t\bigg[& \lambda((\bfu_\eta-\bsfV)\cdot\nabla\bfu+\dot{\theta}\bfu^{\perp},\bfpsi)+{\frac1\varpi}\dot{\theta}\bfxi^{\perp}\cdot\hat{\bfrho}\\&
+2(\mathbb{D}(\bfu),\mathbb{D}(\bfpsi))
+{\frac1\varpi}\hat{\bfrho}\cdot\mathbb{B}\cdot\bfdelta+{\frac{\hat{\alpha}k\theta}\tau}-{\frac{c\dot{V}}\varpi}\,\hat{\bfrho}\cdot {\bfb}_\alpha-{\frac{\hat{\alpha}d\dot{{ V}}}\tau}\bigg]\,{\rm d}s\,.
\end{aligned}
\eeq{wf_ivbp}
By \eqref{01-translation+rotation}$_7$-\eqref{01-translation+rotation}$_9$, we also deduce 
$$
\bfdelta(t)-\bfdelta_0=\int_0^t(\bfxi(s)-V(s){\bfb}_\alpha(s)-\dot{\theta}(s)\bfdelta^\perp(s))\,{\rm d}s \qquad \text{and} \qquad \theta(t)-\theta_0=\int^t_0\omega(s)\,{\rm d}s\,.
$$

With such a weak form at hand, we give the definition of weak solution to \eqref{01-translation+rotation}.
\begin{definition}
The quintuple $(\bfu,\bfxi,\bfdelta,\omega, \theta)$ is a weak solution to \eqref{01-translation+rotation}-\eqref{01-translation+initial-cond} if, for all $t>0$:
\begin{enumerate}[(i)]
\item $\bfxi\in C([0,t];\mathbb{R}^3)$, $\bfdelta\in C^1([0,t];\mathbb{R}^3)$, $\omega\in C([0,t];\mathbb{R})$, $\theta\in C^1([0,t];\mathbb{R})$ ;
\item $\bfu\in C_w([0,t];L^2(\Omega_R))\cap L^2(0,t;\mathcal{D}^{1,2}(\Omega_R))$\,, with $\bfu(\bfx,t)|_{\partial \Omega}=\bfxi(t)+\omega(t)\bfx^\perp$ ;
\item[(iii)]  $(\bfu,\bfxi,\bfdelta,\omega, \theta)$ satisfies \eqref{wf_ivbp} for all $\bfpsi\in \mathcal{D}^{1,2}(B_R)$\,. 
\end{enumerate}
\end{definition}

In the next lemma we shall prove well-posedness of \eqref{01-translation+rotation}-\eqref{01-translation+initial-cond} in the class of weak solutions. The existence part of the proof follows from the Galerkin method combined with a number of appropriate energy estimates.  Since the procedure is rather standard, we limit ourselves  to derive only the estimates, referring to  \cite{BoGa,Ana-Paolo_Lady} for the missing technical details. However, we shall provide a full proof of the continuous dependence on the initial data which then obviously yields uniqueness.

We recall that the total energy of the system is defined by 
\be
E(t) := \frac1{2}\left(\|\bfu(t)\|_{2,\Omega_R}^2+ \frac1{\varpi}\left(|\bfxi(t)|^2 + \bfdelta(t)\cdot\mathbb B(t)\cdot\bfdelta(t)\right)+\frac1{\tau}\left(|\dot\theta(t)|^2+k|\theta(t)|^2\right)\right).
\eeq{energy_functional}
In the following results, we emphasize that the standard energy estimate of assertion {\it (iv)}, the constants $C_i$, $i=1,2,3$, do not depend on the radius $R$ nor on the mollification parameter $\eta$. 
\begin{lemma}\label{Lemma_glob_ex}
Let $V\in H^1(0,T;\mathbb{R})$, $(\bfxi_0,\bfdelta_0,\theta_0,\omega_0)\in \mathbb{R}^3\times \mathbb{R}^3\times\mathbb{R}\times\mathbb{R}$ and let $\bfu_0\in \mathcal{H}(B_R)$ be such that ${\bfu_0}_{|\Omega_0}=\bfxi_0 + \omega_0 \bfx^\perp$. Then, there exists a unique weak solution $(\bfu,\bfxi,\bfdelta,\omega,\theta)$ to \eqref{01-translation+rotation}-\eqref{01-translation+initial-cond} such that 
\begin{enumerate}[(i)]
	\item for all $\sigma>0$ and all $t>\sigma>0$, 
	$$
	\begin{aligned}
&	\bfu\in H^1(\sigma,t;L^2(\Omega_R))\cap L^2(\sigma,t,W^{2,2}(\Omega_R))\,	,\\& \bfxi\in H^1(\sigma,t;\mathbb{R}^3)\,,\quad \bfdelta \in W^{2,2}(\sigma,t;\mathbb{R}^3)\,, \quad \theta\in H^1(\sigma,t;\mathbb{R})\,,\quad \omega \in W^{2,2}(\sigma,t;\mathbb{R})\, ;
	\end{aligned}
	$$
	\item for all $\sigma>0$ and all $t>\sigma>0$, $\bfu\in C([\sigma,t];\mathcal{D}^{1,2}(\Omega_R))$,  and there exists $p\in L^2(\sigma,t;H^1(\Omega_R)) $ such that $(\bfu,p,\bfxi,\bfdelta,\omega,\theta)$ satisfies \eqref{01-translation+rotation}$_1$-\eqref{01-translation+rotation}$_2$ a.a. in $\Omega_R\times (0,t)$ and \eqref{01-translation+rotation}$_5$-\eqref{01-translation+rotation}$_8$ a.e. in $(0,t)$;
	\item the initial conditions $(\bfu_0,\bfxi_0,\bfdelta_0,\omega_0,\theta_0)$ are attained by $\bfxi,\bfdelta,\omega,\theta$ in the sense of pointwise continuity and by $\bfu$ in the $L^2$-sense, i.e. 
	$$
	\lim_{t\to0}\|\bfu(t)-\bfu_0\|_{2,\Omega_R}=0\,;
	$$
	\item for some $C_1$, $C_2$, $C_3>0$ independent of $R,\eta$, and some $C>0$ depending on the data, $\eta,R$ and $t$, the following energy estimates
		\be
	\begin{aligned}
	&E(t) +C_3\Int0t{\rm e}^{C_1(t-s)}(\|\nabla \bfu(s)\|^2_{2,\Omega_R}+|\bfxi(s)|^2+|\dot{\theta}(s)|^2){\rm d}s \\& \hspace{10mm}\le  E(0){\rm e}^{C_1t} + C_2 \Int0t{\rm e}^{C_1(t-s)}\,(|V(s)|^2+|\dot{V}(s)|^2)\, {\rm d}s\,,\\& \max_{[\sigma,t]}\|\nabla \bfu(s)\|_{2,\Omega_R}\le C\,,\quad \forall\,t>\sigma\,,
	\end{aligned}
	\eeq{energie}
	hold. Moreover, the solution depends continuously on the initial data in the norm $E^{\tfrac{1}{2}}$.
\end{enumerate}
\end{lemma}

\begin{proof} As said before, in order to show existence, we will formally prove several estimates that, once combined with Galerkin method, will lead to the existence of a weak solution with the stated property. The argument can be made completely rigorous by proceeding exactly as in \cite[Lemma 3.1]{BoGa} and \cite{Ana-Paolo_Lady}, to which we refer the reader for the missing details. \label{sec:estimates} 
\medbreak 
\noindent{\bf Energy estimate.} We want to show that there exist positive constants $C_1$, $C_2$, $C_3$, independent of $R$ and $\eta$, such that
\be
\begin{aligned}
E(t) +C_3\Int0t{\rm e}^{C_1(t-s)}(\|\nabla \bfu(s)\|^2_2+|\bfxi(s)|^2&+|\dot{\theta}(s)|^2)\,{\rm d}s \\&\le  E(0){\rm e}^{C_1t} + C_2 \Int0t{\rm e}^{C_1(t-s)}\,(|V(s)|^2+|\dot{V}(s)|^2)\, {\rm d}s\,.
\end{aligned}
\eeq{energia}
We first choose $\bfpsi=\bfu$ in \eqref{wf_ivbp}.  Observing that 
\[
\begin{split}
\bfxi\cdot \mathbb B\cdot \bfdelta = (\dot\bfdelta+\dot{\theta}\bfdelta^\perp+V\bfb_\alpha)\cdot \mathbb B\cdot\bfdelta = \ \frac{1}2\ode{}t({\bfdelta}\cdot\B\cdot\bfdelta) +  V\bfb_\alpha\cdot\B\cdot\bfdelta\,,
\end{split}
\]
and considering the energy in \eqref{energy_functional}, we deduce that 
\be
\ode{}tE + 2\|\mathbb D(\bfu)\|_2^2 = {\frac{c\dot V}\varpi} \bfxi \cdot  \bfb_\alpha -\frac{V}\varpi \bfb_\alpha\cdot\B\cdot\bfdelta + \frac{d\dot\theta\dot V}\tau\,.
\eeq{en_eq}
We now recall the trace inequality
\be
2\|\mathbb D(\bfu)\|_2^2\ge \frac12\|\nabla\bfu\|_2^2+\kappa(|\bfxi|^2+|\dot\theta|^2)\,,
\eeq{3.14}
where $\kappa>0$ depends on $\Omega_0$ only, see e.g. \cite[Lemma 3.1]{Ana-Paolo_Lady}, \cite[Lemma 4.9]{Gah}. 
By  Cauchy-Schwarz and  Young inequalities, with the help of \eqref{B_bound} we get
$$
{\frac{c\dot V}\varpi}\bfxi \cdot  \bfb_\alpha -\frac{V}\varpi\bfb_\alpha\cdot\B\cdot\bfdelta + \frac{d\dot\theta \dot V}\tau\le \frac{\kappa}2(|\bfxi|^2+|\dot\theta|^2)+\frac1{2\varpi}{|\bfdelta|^2}+C_2(|V(t)|^2+|\dot{V}(t)|^2)\,,
$$
where $C_2>0$ depends on  $\mathbb{A},\Omega_0$, $\varpi$ and $\tau$.
Thus, combining the latter with \eqref{3.14}, from \eqref{en_eq} 
 and \eqref{B_bound} we deduce 
that 
\be\ode{}tE(t) +\frac12\|\nabla\bfu(t)\|_2^2 + \frac\kappa2(|\bfxi(t)|^2+|\dot\theta(t)|^2) \le C_1 E(t) + C_2\,(|V(t)|^2+|\dot{V}(t)|^2)\,,\eeq{EnIN}
where the constant $C_1$ depends only on the stiffness matrix $\A$. The relation in \eqref{energia} then easily follows after integrating of the differential inequality \eqref{EnIN}. 
\medbreak
\noindent{\bf Time-weighted estimate for the gradient.} We next prove that for all $\tau>0$,
\be
\sup_{t\in [0,\tau]}\left(t\,\|\nabla\bfu(t)\|_2\right)+\int_0^\tau({t}\|\bfu(t)\|_{2,2}^2)\,{\rm d}t\le H_1(\tau)\,,
\eeq{final2} 
where the right-hand-side $H_1$ depends on the $L^2$-norm of the initial data, the $H^1$-norm of $V$, $\Omega_0$, $\lambda$, $\alpha$, $k$, $\tau$, $\varpi$, $\mathbb{A}$, $\eta$ and $R$.
To get this estimate, we take $\bfpsi=-t\, \Div\mathbb T(\bfu,p)$ as multiplier in $\eqref{01-translation+rotation}_1$. 
By a straightforward computation, we show
\be
\begin{aligned}
\int_{\Omega_{R}}\partial_t \bfu \cdot (-t\, \text{div}\mathbb{T}(\bfu,p))&=-\int_{\Omega_R}\partial_t(\bfu\cdot t\, \text{div}\mathbb{T}(\bfu,p)) +\int_{\Omega_R} \bfu \cdot \text{div}\mathbb{T}(\bfu,p)+t\,\int_{\Omega_R}\bfu\cdot \text{div}(\partial_t \mathbb{T}(\bfu,p))
\\&= t\left(\frac{d}{dt}\left(\|\nabla \bfu\|^2_2-{\bfxi}\cdot{\mathbf \Sigma}-\dot{\theta}\sigma_1\right)-\int_{\Omega_R}\nabla \bfu: \nabla \partial_t \bfu +\int_{\partial \Omega_R}\bfu \cdot (\partial_t\mathbb{T}(\bfu,p) )\cdot \bfn\right)\\&=\frac{t}{2}\frac{d}{dt}\|\nabla \bfu\|^2_2-t\frac{d}{dt}(\bfxi\cdot {\mathbf \Sigma}+\dot{\theta}\sigma_1)+t\,{\bfxi}\cdot \partial_t {\mathbf \Sigma}+t\, \dot{\theta}\partial_t \sigma_1
\\&=\frac{1}2 \frac{d}{dt}\left(t\|\nabla \bfu\|^2_2\right)-\frac{1}{2}\|\nabla \bfu\|^2_2-t \, \dot{\bfxi}\cdot {\mathbf{\Sigma}} - t\, \dot{\omega}\sigma_1\,,
\end{aligned}
\eeq{dtu_div}
where
$$
\bfSigma:=\int_{\partial\Omega}\mathbb T(\bfu,p)\cdot\bfn\,,\ \ \sigma_1:=\bfe_1\cdot\int_{\partial\Omega}\bfx\times\mathbb T(\bfu,p)\cdot\bfn\,.
$$	
Thus,  integrating by parts over $\Omega_R$ and using \eqref{dtu_div}, and $\eqref{01-translation+rotation}_5$--$\eqref{01-translation+rotation}_8$, we infer
\be\begin{split}
\frac{1}{2}\ode{}t\Big( &t\|\nabla \bfu\|_2^2\Big)+{t}\|\Div\mathbb T(\bfu,p)\|_2^2 -\frac12\|\nabla\bfu\|_2^2+\varpi\,t\,|\bfSigma|^2 +\tau\,t\,|\sigma_1|^2 \\ \medskip
= &\ t\,\lambda\left(\left(\bfu_\eta\cdot\nabla\bfu,\Div\mathbb T(\bfu,p)\right)-\left(\bfxi\cdot\nabla\bfu,\Div\mathbb T(\bfu,p)\right)-\dot{\theta}\left(\bfU,\Div\mathbb T(\bfu,p)\right)\right)
\\
&\ -t\,\left(\bfSigma\cdot\B\cdot\bfdelta{+}\dot\theta\,\bfSigma\cdot\bfxi^\perp- c\,\dot V\, \bfSigma\cdot \bfb_\alpha {+} k\theta\sigma_1-d\, \dot V \sigma_1\right)\,,
 \end{split}
\eeq{3.18_}
where we set $\bfU:=\bfx^\perp \cdot \nabla \bfu-\bfu^\perp$.
We now estimate the right-hand side of \eqref{3.18_} term by term. To this end, denote by ${\sf D}={\sf D}(t)$ the right-hand side of \eqref{energia}. 
In the estimates that follow, we will tacitly make use of the Cauchy-Schwarz inequality, the Young inequality and \eqref{3.14} several times. From \eqref{moly}   and \eqref{energia}, it follows that
\be
\begin{split}
\,|\lambda\left((\bfu)_\eta\cdot\nabla\bfu,\Div\mathbb T(\bfu,p)\right)| & \le  \frac{3c^2_\eta\lambda^2}2\,\|\bfu\|_2^2\|\nabla\bfu\|_2^2+{\frac{1}{ 6}}\|\Div\mathbb T(\bfu,p)\|_2^2\\ 
&\le C\,{\sf D}(t)\|\nabla\bfu\|_2^2+{\frac{1}{ 6}}\|\Div\mathbb T(\bfu,p)\|_2^2\,,
 \end{split}
\eeq{3.20}
where, here and in the rest of the proof, $C$ is a positive constant depending, at most, on $\Omega,\eta,R$ and the physical parameters.
Similarly, again by \eqref{energia}, we show
\be
\begin{split}
|\lambda\,(\bfxi\cdot \nabla \bfu,\Div\mathbb T(\bfu,p))|& \le \frac{3\lambda^2|\bfxi|^2}2\, \|\nabla \bfu\|^2_2+\frac{1}{ 6}\|\Div\mathbb T(\bfu,p)\|^2_2\\
& \le C\,{\sf D}(t)\|\nabla\bfu\|_2^2+{\frac{1}{ 6}}\|\Div\mathbb T(\bfu,p)\|_2^2\,,
\end{split}
\eeq{NaFi}
and 
\be
\begin{split}
 |\lambda\dot{\theta}\left(\bfU,\Div\mathbb T(\bfu,p)\right)| &\le \frac{3R^2{\lambda^2}|\dot{\theta}|^2}2\, (\|\nabla \bfu\|^2_2+\|\bfu\|^2_2)+\frac{1}{ 6}\|\Div\mathbb T(\bfu,p)\|^2_2\\ 
 &\le C\,{\sf D}(t)\|\nabla\bfu\|_2^2+{\frac{1}{ 6}}\|\Div\mathbb T(\bfu,p)\|_2^2\,,
\end{split}
\eeq{NaFi1}
where, in the last step, we also used Poincar\'e inequality.
Moreover, there holds
$$
|{\bf\Sigma}\cdot\mathbb{B}\cdot{\bfdelta}+\dot{\theta}{\bf \Sigma}\cdot \bfxi^\perp-c\,\dot{V}{\bf \Sigma}\cdot {\bfb}_\alpha|\le\frac{\varpi}{2}|{\bf \Sigma}|^2+ \frac{3}{2\varpi}\left(|\dot{\theta}|^2\,|\bfxi|^2+{c}|{\bfb}_\alpha|^2|\dot{V}|^2+ {c_1}{}|\bfdelta|^2\right)\,,
$$
so that, by \eqref{energia} and \eqref{B_bound},
\be 
|{\bf\Sigma}\cdot\mathbb{B}\cdot{\bfdelta}+\dot{\theta}{\bf \Sigma}\cdot \bfxi^\perp-c\,\dot{V}{\bf \Sigma}\cdot {\bfb}_\alpha|\le \frac{\varpi}{2}|{\bf \Sigma}|^2+C\,\big({\sf D}^2(t)+|\dot{V}|^2+|\bfdelta|^2\big)\,.
\eeq{NaFi2}
Likewise,
\be
|k\theta\sigma_1-d\, \dot V \sigma_1|\le \frac{\tau}{2}|\sigma_1|^2+\frac{1}{\tau}\left(k^2|\theta|^2+d^2|\dot{V}|^2\right)\,.
\eeq{NaFi3}
Therefore, combining \eqref{dtu_div} with \eqref{3.18_}--\eqref{NaFi3} we get
\be
\frac{d}{dt}\Big( t\|\nabla \bfu\|_2^2\Big)+ {t}\|\Div\mathbb T(\bfu,p)\|_2^2 +{\varpi\,t}\,|\bfSigma|^2 +{\tau\,t}\,|\sigma_1|^2
\le C\,{\sf D}(t)t\|\nabla\bfu\|_2^2+G(t)
\eeq{final}
where $$
G(t):=C\,t\big({\sf D}^2(t)+|\bfdelta(t)|^2+|\theta(t)|^2+|\dot{V}(t)|^2\big)\,.
$$
From the energy estimate \eqref{energia},  it follows that, for  arbitrary $\tau>0$, 
$$
\int_0^\tau({\sf D}(t)+G(t))\,{\rm d}t\le F(\tau)\,,
$$
where $F(\tau)$ is a smooth function of $\tau$, depending on the norm of the initial data, the $H^1$-norm of $V$, $\Omega_0, \lambda,\alpha,k,\tau,\varpi,\mathbb{A}$, $\eta$, and $R$.  
As a result, using Gronwall's lemma in \eqref{final} entails, in particular,
\be
\sup_{t\in [0,\tau]}\left(t\,\|\nabla\bfu(t)\|_2\right)+\int_0^\tau({t}\|\text{div}\mathbb{T}(\bfu(t),p(t))\|_2^2)\,{\rm d}t\le H_1(\tau)\,,\ \ \mbox{for all $\tau>0$}\,,
\eeq{3.23}
where $H_1$ has the same property as $F$.  This proves \eqref{final2}.
\medbreak

\noindent{\bf Time-weighted estimate for the time derivative.} We claim 
\be 
\int_0^\tau t(\|\partial_t\bfu(t)\|_{2}^2+|\dot{\bfxi}(t)|^2+|\dot{\omega}(t)|^2)\,{\rm d}t\le H_2(\tau)\,,\ \ \mbox{for all $\tau>0$}\,,
\eeq{3.25}
where $H_2$ has the same property as $H_1$.
To prove our claim, it is enough to mimic the formal choice of multiplying by $\bfpsi=t\,\partial_t\bfu$ both sides of \eqref{01-translation+rotation}$_1$ and integrating by parts over $\Omega_R$ as necessary.  Using \eqref{01-translation+rotation}$_5$-\eqref{01-translation+rotation}$_7$ and taking \eqref{dtu_div} into account, we deduce that
\be\begin{split}
	\frac12\ode{}t\Big( &t\|\nabla \bfu\|_2^2\Big)+ \,{t}\|\partial_t\bfu\|_2^2 -\frac12\|\nabla\bfu\|_2^2+\frac{t}{\varpi}\,|\dot{\bfxi}|^2 +\frac{t}{\tau}|\dot{\omega}|^2 \\ \medskip
	= &\ -\lambda\,t\left(\left(\bfu_\eta\cdot\nabla\bfu,\partial_t\bfu\right)-\left(\bfxi\cdot\nabla\bfu,\partial_t\bfu\right)-\dot{\theta}\left(\bfU,\partial_t\bfu\right)\right)
	\\
	&\ -t\,\left(\frac{k\theta \dot{\omega}}{\tau}-\frac{d\dot{\omega}\dot{V}}{\tau}+\frac{\dot{\theta}}{\varpi}\dot{\bfxi}\cdot\bfxi^\perp+\frac1{\varpi}{\dot{\bfxi}}\cdot\mathbb{B}(t)\cdot \bfdelta-\frac{c\dot{V}}{\varpi}\,\dot{\bfxi}\cdot {\bfb}_\alpha\right)\,.
\end{split}
\eeq{3.18}
The first three terms on the right-hand side of \eqref{3.18} can be increased exactly as in \eqref{3.20}--\eqref{NaFi1} with the replacement $\Div\mathbb T(\bfu,p)\to \partial_t\bfu$. Moreover, the remaining terms are estimated in a way similar to \eqref{NaFi2}--\eqref{NaFi3}.
As a result, the inequality in \eqref{3.25} is derived by
 arguing in the same way as we did for the gradient estimate.

\medbreak
\noindent{\bf Continuous dependence on the initial data.}
Uniqueness will be a consequence of the continuous dependence on the initial data. Let $(\bfu_i,\bfxi_i,\bfdelta_i,\omega_i,\theta_i)$, $i=1,2$, be two weak solutions corresponding to the same $V$. We use obvious notations for all quantities depending on the solution indexed by $i$.
Setting $\bfu=\bfu_1-\bfu_2, \bfxi=\bfxi_1-\bfxi_2,\bfdelta=\bfdelta_1-\bfdelta_2,\omega=\omega_1-\omega_2,\theta=\theta_1-\theta_2$, $\bsfV=\bsfV_1-\bsfV_2$ and so on, we deduce for  arbitrary $\tau>\sigma>0$
\be\left\{\ba{l}\medskip\left.\ba{r}\medskip
†\partial_t\bfu+\lambda\left(((\bfu+\bfu_2)_\eta-(\bsfV+\bsfV_2))\cdot\nabla\bfu+(\bfu_\eta-\bsfV)\cdot \nabla \bfu_2+(\dot{\theta}+\dot{\theta}_2)\bfu^\perp +\dot{\theta}\bfu_2^\perp\right)\\= \Div\mathbb T(\bfu,p)\medskip\\
\Div\bfu=0\ea\right\}\ \ \mbox{in $\Omega_R\times[\sigma,\tau]$}\,,\\ \medskip
\ \ \bfu({\bfx},t)=\bsfV({\bfx},t):=\bfxi(t)+\omega(t){\bfx}^\perp\,, \ \mbox{at $\partial\Omega\times[\sigma,\tau]$}\,;\ \\ \medskip
\ \ \bfu(\bfx,t)={\bf0}\,,\ \text{at}\,\, \partial B_R\times[\sigma,\tau]\,,\\
\ \ 
\medskip\medskip\left.\ba{r}\medskip
\dot{\bfxi}+(\dot\theta+\dot\theta_2)\bfxi^\perp+\dot\theta\bfxi_2^\perp+(\mathbb B_1(t)-\mathbb B_2(t))\cdot \bfdelta+\mathbb{B}_2\cdot\bfdelta+(\mathbb B_1(t)-\mathbb B_2(t))\cdot\bfdelta_2\\[3pt]+\varpi\Int{\partial\Omega}{}\mathbb T(\bfu,p)\cdot\bfn= c\, \dot V(t) ({\bfb}_{\alpha, 1}(t)-{\bfb}_{\alpha, 2}(t))\medskip\medskip
\\ \medskip 
\dot{\bfdelta}+(\dot\theta+\dot\theta_2)\bfdelta^\perp+\dot\theta\bfdelta^\perp_2={\bfxi-V(t)(({\bfb}_{\alpha, 1}(t)-{\bfb}_{\alpha, 2}(t))}
\\ \medskip
\dot{\omega}+k\theta+\tau\,\bfe_1\cdot\Int{\partial\Omega}{}\bfx\times\mathbb T(\bfu,{ p})\cdot\bfn=0\, \\
\medskip
\dot{\theta}={\omega}
\ea\right\}\ \ \mbox{in $\real$\, ,}
\ea\right.
\eeq{01-translation+rotation_difference}
for some $p\in L^2(\sigma,\tau;H^1(\Omega_R))$, and initial conditions 
$$
\begin{aligned}
&\bfu(\bfx,0)=\bfu_1(\bfx,0)-\bfu_2(\bfx,0)\,,\qquad  \bfxi(0)=\bfxi_1(0)-\bfxi_2(0)\,,\quad \bfdelta(0)=\bfdelta_1(0)-\bfdelta_2(0)\,,\\& \omega(0)=\omega_1(0)-\omega_2(0)\,,\quad \theta(0)=\theta_1(0)-\theta_2(0)\,.
\end{aligned}
$$
Testing \eqref{01-translation+rotation_difference}$_1$ with $\bfu$, integrating by parts over $\Omega_R$ and taking in account all other identites in  \eqref{01-translation+rotation_difference}, we get
\be
\begin{aligned}
\frac{d}{dt}\tilde{E}+2\|\mathbb{D}(\bfu)\|^2_2=&-\lambda((\bfu_\eta-\bsfV)\cdot \nabla\bfu_2,\bfu)-\lambda\dot\theta(\bfu_2^\perp,\bfu)+\frac{c\dot{V}}{\varpi}\bfxi\cdot({\bfb}_{\alpha, 1}-{\bfb}_{\alpha, 2})-\frac{\dot\theta}{\varpi}\bfxi\cdot\bfxi_2^\perp\\&-\frac{1}{\varpi}\bfxi\cdot{(\mathbb{B}_1-\mathbb{B}_2)}\cdot {\bfdelta}-\frac{1}{\varpi}\bfxi\cdot{\mathbb{B}_2}\cdot {\bfdelta}-\frac{1}{\varpi}\bfxi\cdot{(\mathbb{B}_1-\mathbb{B}_2)}\cdot {\bfdelta}_2\,,
\end{aligned}
\eeq{eq_differenza}
with
$$\tilde{E}(t):= \frac1{2}\|\bfu(t)\|_2^2+ \frac1{2\varpi}|\dot\bfxi(t)|^2+\frac{1}{2}|\bfdelta(t)|^2+\frac1{2\tau}\left(|\dot\theta(t)|^2+k|\theta(t)|^2\right).$$
In view of \eqref{B_bound}, $\tilde{E}$ is equivalent to $E$ given in \eqref{energy_functional}. 
On the other hand, testing \eqref{01-translation+rotation_difference}$_7$ with $\bfdelta$, we deduce
\be
\frac{1}{2}\frac{d}{dt}|\bfdelta|^2=\bfxi\cdot\bfdelta-V({\bfb}_{\alpha, 1}-{\bfb}_{\alpha, 2})\cdot\bfdelta-\dot{\theta}\bfdelta_2^\perp\cdot \bfdelta\,.
\eeq{eq_differenza2}
We next estimate each term on the right-hand side of \eqref{eq_differenza}-\eqref{eq_differenza2}.
Using \eqref{moly}, we infer
\be
\begin{aligned}
&\left|\lambda((\bfu_\eta-\bsfV)\cdot \nabla\bfu_2,\bfu)+\lambda(\dot\theta\bfu_2^\perp,\bfu)\right|\\&\le  c_\eta \lambda\|\nabla \bfu_2\|_2\|\bfu\|_2^2+\lambda|\bsfV|\|\nabla \bfu_2\|_2\|\bfu\|_2+\lambda |\dot{\theta}\|\bfu_2\|_2\|\bfu\|_2\\&\le c_\eta\lambda \|\nabla \bfu_2\|_2\,\|\bfu\|_2^2+\frac{{\lambda}^2}2\left(\|\nabla \bfu_2\|^2_2+\|\bfu_2\|^2_2\right)\|\bfu\|_2^2+\frac{|\bfxi|^2}{2}+c_1{|\dot\theta|^2}\,,
\end{aligned}
\eeq{first_est}
where $c_1>0$ depends on $R$. We then remark that
\be
\begin{aligned}
2\left|{c\dot{V}}\bfxi\cdot({\bfb}_{\alpha 1}-{\bfb}_{\alpha 2})-{\dot\theta}\bfxi\cdot\bfxi_2^\perp\right|&\le 
{c^2|\dot{V}|^2}|{\bfb}_{\alpha, 1}-{\bfb}_{\alpha, 2}|^2+(1+|\bfxi_2|^2){|\bfxi|^2}+{|\dot{\theta}|^2}\\&
\le {c_2c^2}|\dot{V}|^2{|\theta|^2}+(1+|\bfxi_2|^2){|\bfxi|^2}+{|\dot{\theta}|^2}\,,
\end{aligned}
\eeq{sec_est}
and
\be
\begin{aligned}
&2\Big|\bfxi\cdot{(\mathbb{B}_1-\mathbb{B}_2)}\cdot {\bfdelta}-\bfxi\cdot{\mathbb{B}_2}\cdot {\bfdelta}-\bfxi\cdot{(\mathbb{B}_1-\mathbb{B}_2)}\cdot {\bfdelta}_2\Big|\\&
\le{3}|\bfxi|^2+|(\mathbb{B}_1-\mathbb{B}_2)\cdot \bfdelta|^2+|\mathbb{B}_2\cdot \bfdelta|^2+|(\mathbb{B}_2-\mathbb{B}_1)\cdot \bfdelta_2|^2\\&
\le {3}|\bfxi|^2+c_3\left(|\bfdelta|^2+|\bfdelta_2|^2|\theta|^2\right),
\end{aligned}
\eeq{third_est}
where $c_2>0$ depends on $\alpha$, while $c_3>0$ depends on $\mathbb{A}$. Finally
\be
2\left|\bfxi\cdot\bfdelta-V({\bfb}_{\alpha 1}-{\bfb}_{\alpha 2})\cdot\bfdelta-\dot{\theta}\bfdelta_2^\perp\cdot \bfdelta\right|\le{3} {|\bfdelta|^2}+{|\bfxi|^2}+c_2|V|^2 |\theta|^2+{|\bfdelta_2|^2}|\dot{\theta}|^2\,.
\eeq{last_est}
Summing \eqref{eq_differenza} with \eqref{eq_differenza2} and combining \eqref{first_est}-\eqref{sec_est}-\eqref{third_est}-\eqref{last_est}, we infer that
$$
\begin{aligned}
\frac{d}{dt}\tilde{E}(t) \le 
g(t)\,\tilde{E}(t)\,,
\end{aligned}
$$
where $g(t)$ depends on $\bfu_2(t),\bfxi_2(t),\bfdelta_2(t),\tau,\varpi,\lambda,\mathbb{A},k,\alpha,V,\Omega_0$.
Gronwall's lemma yields 
$$
\tilde{E}(\tau)\le \tilde{E}(\sigma)\,\text{exp}\left(\int_0^\tau g(s)\,{d} s\right).
$$
Letting $\sigma\to 0$, we get
$$
\lim_{\sigma\to 0}\tilde{E}(\sigma)=\tilde{E}(0)
$$
from assertions $(i)$-$(ii)$-$(iii)$. This concludes the proof of the continuous dependence. 
\end{proof}

\section{Local dissipation of the total energy}\label{Sec:Haraux_trick}
In Lemma \ref{Lemma_glob_ex}, we showed that the total energy $E=E(t)$ of the coupled system defined in \eqref{energy_functional} is bounded on any finite time interval. However, as is well known \cite{PaMa}, this is not enough to ensure that the Poincar\'e map, bringing initial data to the corresponding solution at time $T$, has a fixed point or, in other words, to secure the existence of a $T$-periodic solution. A stronger requirement sufficient to show the latter is that, at least locally, namely, in the time interval $[0,T]$, $E(t)$ decreases, in absence of external forces, with a {\em definite} decay rate, when  $E(0)$ is below a given value. The aim of this section is to show the validity of such a property.   
\par
To this end, we start by constructing a solenoidal extension of $\bfdelta+\theta \bfx^\perp$. For $(\bfx,t)\in\mathbb{R}^3\times(0,\infty)$, we define
$$
\bfH(\bfx,t):=\text{curl}(\psi(|\bfx|)\mathsf{H}(\bfx,t))\,,
$$
where $\psi$ is a smooth cut-off function, that equals $1$ in a neighborhood of $\Omega_0$ and $0$ for $|\bfx|\ge 2R_*$ while
$$
{\mathbf{\mathsf{H}}}(\bfx,t):=\left(\delta_2(t)x_3-\frac{x_3^2}{2}\theta(t)\right)\bfe_1+\left(\delta_3(t)x_1+x_2x_1\theta(t)\right)\bfe_2+x_2\delta_1(t)\bfe_3.
$$
Obviously, $\text{div} \bfH(\bfx,t)=0$ for $(\bfx,t)\in \Omega_R\times (0,\infty)$ and $\bfH(\bfx,t)=\bfdelta(t)+\theta(t)\bfx^\perp$ for $(\bfx,t)\in \partial \Omega_0\times (0,\infty)$.
One checks easily that there exists $c_1>0$ that depends on $\psi$ (and therefore on $R_*$) such that 
\be
\begin{aligned}
&\sup_{\bfx\in \Omega_R}|\bfH(\bfx,t)|\le c_1(|\bfdelta(t)|+|\theta(t)|)\quad t\in(0,\infty)\,,\\&
\sup_{\bfx\in \Omega_R}|\bfH_t(\bfx,t)|\le c_1(|\dot{\bfdelta}(t)|+|\dot{\theta}(t)|) \quad t\in(0,\infty)\,.
\end{aligned}
\eeq{stime_H}
For $\zeta \in (0,\infty)$, we then define a perturbation of the energy functional 
$$
E_\zeta(\bfu,\bfxi, \bfdelta,\omega,\theta)= E(\bfu,\bfxi, \bfdelta,\omega,\theta)+2\zeta \left[ (\bfu, \bfH)+\frac1\varpi\bfxi\cdot \bfdelta+\frac{\omega\theta}{\tau}\right]\,,
$$
where $E$ is given in \eqref{energy_functional}. We claim that there exists $\zeta_1=\zeta_1(\varpi,k,\tau,\mathbb{A},R_*)$ such that, if $\zeta\le \zeta_1$, then 
\be
\frac{1}{2}E\le E_\zeta\le \frac{3}{2}E\,.
\eeq{E_zeta_dis}
Indeed, the Cauchy-Schwarz inequality combined with the Young inequality and \eqref{stime_H} give
$$
\begin{aligned}
\zeta\left| (\bfu, \bfH)+\frac{1}{\varpi}\bfxi\cdot \bfdelta+\frac{\omega\theta}{\tau}\right|&
\le \frac{1}{8}\left(\|\bfu\|^2_2+\frac{|\bfxi|^2}{\varpi}+\frac{|\omega|^2}{\tau}\right)+{2\zeta^2}\left(\|\bfH\|_2^2+\frac{|\bfdelta|^2}{\varpi}+\frac{|\theta|^2}{\tau}\right)\\&
\le \frac{1}{8}\left(\|\bfu\|^2_2+\frac{|\bfxi|^2}{\varpi}+\frac{|\omega|^2}{\tau}\right)+2\zeta^2\left( c_1+\frac{1}{\varpi} \right)|\bfdelta|^2+2\zeta^2\left(c_1+\frac{1}{\tau}\right)|\theta|^2\,,
\end{aligned}
$$
where $c_1>0$ still denotes a constant depending basically on $R_*$. 
Using \eqref{B_bound} and choosing $\zeta_1$ sufficiently small allows to deduce \eqref{E_zeta_dis}. In the next lemma we prove the desired result on local dissipation.

\begin{lemma}\label{Lemma_dissipation}
Let $V\in H^1(0,T;\mathbb{R})$  and let $(\bfu,\bfxi,\bfdelta,\omega,\theta)$ be the corresponding unique solution to \eqref{01-translation+rotation} given in Lemma \ref{Lemma_glob_ex}.
There exist $\rho_0>0$ and $\zeta_0\in(0,1]$ depending $T,\mathcal{V},R,\Omega_0,\varpi,\tau,\alpha,\mathbb{A},k,\lambda$, such that for all $\eta\in (0,\eta_0)$, $E_{\zeta_0}(0)\le \rho_0$ implies $E_{\zeta_0}(T)\le \rho_0$ as well. Moreover, one has
	\be
 E_{\zeta_0}(t)+\zeta_0\int_0^t\left({\frac1\varpi}\bfdelta(s)\cdot\mathbb{B}(s)\cdot\bfdelta(t)+{\frac{k}\tau}|\theta(t)|^2\,\right){\rm d}s\le  E_{\zeta_0}(0)+C_0\mathcal{V}\,,\ \ \mbox{for all $t\in[0,T]$\,,}
	\eeq{claim2}
where $C_0>0$ depends on $R,\Omega_0,\varpi,\tau,\alpha,\mathbb{A},k$, but not on $\eta$. 
\end{lemma}
\begin{proof}
For sufficiently small $\sigma>0$, we consider \eqref{01-translation+rotation} with $t\in [\sigma, T]
$.
 Since $\bfu\in C([\sigma,t];\mathcal{D}^{1,2}(\Omega_R))$, we can test \eqref{01-translation+rotation}$_1$ by $\bfu$.
Arguing as in Section \ref{sec:estimates}, we arrive at
\be\ode{}tE + 2\|\mathbb D(\bfu)\|_2^2 = {\frac{c\dot V }\varpi}\bfxi \cdot  \bfb_\alpha -\frac{V}\varpi\,\bfb_\alpha\cdot\B\cdot\bfdelta + \frac{d\dot\theta\, \dot V}\tau \,.
\eeq{molt_u}
On the other hand, testing \eqref{01-translation+rotation}$_1$ by $\bfH$, and taking \eqref{01-translation+rotation}$_5$-
\eqref{01-translation+rotation}$_7$ into account, we get (with $\bfH_t\equiv\partial_t\bfH$)
\be
\begin{aligned}
	&\frac{d}{dt}\left[(\bfu,\bfH)+\frac1\varpi\bfxi\cdot\bfdelta+{\frac1\tau}\omega\theta\right]+{\frac1\varpi}\bfdelta\cdot\mathbb{B}\cdot\bfdelta+\frac{k}{\tau}|\theta|^2\\ & = (\bfu,\bfH_t)-\lambda((\bfu_\eta-\bsfV)\cdot\nabla \bfu,\bfH)-(\dot{\theta}\bfu^\perp,\bfH)-2(\mathbb{D}(\bfu), \mathbb{D}(\bfH))\\ 
	&\ \quad +\frac{c\dot{V}}{\varpi}\bfdelta\cdot {\bfb}_\alpha-\frac{\dot{\theta}}{\varpi}\bfxi^\perp \cdot\bfdelta +\frac{1}{\varpi}\dot{\bfdelta}\cdot\bfxi+\frac{\theta d\dot{V}+|\dot{\theta}|^2}{\tau}\,.
\end{aligned}
\eeq{molt_H}
Multiplying both sides of \eqref{molt_H} by $2\zeta$ and summing with \eqref{molt_u}, we infer
\be
\begin{aligned}
	&\frac{d}{dt} E_\zeta+2\|\mathbb D(\bfu)\|_2^2+2\zeta\left({\frac1\varpi}\bfdelta\cdot\mathbb{B}\cdot\bfdelta+\frac{ 1}{\tau}\left(k|\theta|^2-|\dot{\theta}|^2\right)\right)\\
	&=2\zeta\left((\bfu,\bfH_t)-\lambda((\bfu_\eta-\bsfV)\cdot\nabla \bfu,\bfH)-(\dot{\theta}\bfu^\perp,\bfH)-2(\mathbb{D}(\bfu), \mathbb{D}(\bfH))\right)\\ 
	&\quad +\frac{2\zeta}{\varpi}\left( c\dot{V}\bfdelta\cdot{\bfb}_\alpha 
	-{\dot{\theta}}\bfxi^\perp \cdot\bfdelta+\dot{\bfdelta}\cdot \bfxi\right)+ \frac{2\zeta d\theta\dot{V}}{\tau}+{\frac{c\, \dot V }\varpi}\bfxi \cdot  \bfb_\alpha -\frac{V}\varpi\bfb_\alpha\cdot\B\cdot\bfdelta + \frac{d\dot\theta \dot V}\tau\,.
\end{aligned}
\eeq{Haraux}
We will proceed analyzing the terms on the right-hand side by gathering together those that, when estimated, generate similar terms. 
Once again, in what follows, the Cauchy-Schwarz and the Young inequality will be repeatedly used, together with \eqref{stime_H}. 
Using the properties of $\bfu_\eta$, \eqref{01-translation+rotation}$_6$, \eqref{3.14} and the Poincar\'e inequality on $\Omega_R$, we deduce
	$$
	\begin{aligned}
		&\left|(\bfu,\bfH_t)+\lambda ((\bfu_\eta-\bsfV)\cdot \nabla \bfu,\bfH)-(\dot{\theta}\bfu^\perp,\bfH)- \frac{\dot{\theta}}{\varpi}\bfxi^\perp\cdot \bfdelta+ \frac{1}{\varpi}\dot{\bfdelta}\cdot \bfxi\right|\\
		&\le \|\bfu\|_2\|\bfH_t\|_2+ \lambda \|\bfH\|_{\infty}\|\bfu_\eta\|_2\|\nabla \bfu\|_2+\lambda|\bfsf{V}|\|\bfH\|_2\|\nabla\bfu\|_2\\
		&\quad+ |\dot{\theta}|\|\bfu\|_2\|\bfH\|_2+\frac{1}\varpi |\dot{\theta}||\bfxi||\bfdelta|+\frac{1}{\varpi}|\dot{\bfdelta}||\bfxi|\\
		&\le c_1(|\bfdelta|+|\theta|)\|\nabla\bfu\|^2_2+ c_2\|\mathbb{D}(\bfu)\|^2_2+\frac{2|\bfxi|^2}{\varpi}+c_3|V|^2\,,
	\end{aligned}
	$$
where $c_1>0$ depends on $\Omega_0, R,\lambda,\varpi$, $c_2>0$ depends on $R,\Omega_0$,  and $c_3>0$ depends on $\alpha$. 
Then, we observe
$$
\begin{aligned}
&\left|4\zeta(\mathbb{D}(\bfu),\mathbb{D}(\bfH))+\frac{2\zeta c\dot{V}}{\varpi} \bfdelta\cdot{\bfb}_\alpha+ \frac{2\zeta d\theta\dot{V}}{\tau}-\frac{V}\varpi\bfb_\alpha\cdot\B\cdot\bfdelta+\frac{c\dot{V}}{\varpi}\bfxi\cdot {\bfb}_\alpha+\frac{d\dot{\theta}\dot{V}}{\tau}\right|\\ 
&\le  
\frac{1}2\|\mathbb{D}(\bfu)\|^2_2+{c_4 \zeta}(|\bfdelta|^2+|\theta|^2)+c_5(|\dot{V}|^2+|V|^2)\,,
\end{aligned}
$$
with $c_4>0$ depending on $R_*,k,\tau,\varpi$, and $c_5>0$ depending on $\mathbb{A},\alpha,k,\tau,R_*,\varpi,\Omega_0$. Finally, we estimate the terms that are not pre-multiplied by $\zeta$ and, using \eqref{3.14},  we infer
$$
\begin{aligned}
\left|\frac{c\dot{V}}{\varpi}\bfxi\cdot {\bfb}_\alpha+\frac{d\dot{\theta}\dot{V}}{\tau}\right|\le \frac{1}2\|\mathbb{D}(\bfu)\|^2_2 +{c_6}|\dot{V}|^2\,,
\end{aligned}
$$
where $c_6>0$ depends on $\varpi,\tau,\Omega_0,\alpha$. 
From \eqref{B_bound}-\eqref{3.14} and the above manipulations, we deduce that there exists $\zeta_2\in (0,\zeta_1)$, depending on 
$\tau,\Omega_0,\varpi,\mathbb{A},R$ such that, if $\zeta\le \zeta_2$, then 
\be
\frac{d}{dt} E_\zeta+\frac{1}2\|\mathbb D(\bfu)\|_2^2+{\frac\zeta\varpi}\bfdelta\cdot\mathbb{B}\cdot\bfdelta+\frac{\zeta k}{\tau}|\theta|^2\le c_1  \zeta (|\bfdelta|+|\theta|)\|\nabla\bfu\|^2_2+c_6(|\dot{V}|^2+|V|^2 )\,.
\eeq{E_zeta}
We next claim that there exists $\zeta_0\in (0,\zeta_2)$ such that, if $E_{\zeta_0}(0)\le \rho_0$, then $E_{\zeta_0}(T)\le \rho_0$ for all $\zeta\le \zeta_0$. To prove this claim, we impose that $E_{\zeta_0}(0)\le \rho_0$ with 
$\rho_0\ge \mathcal{V}$.
Then, \eqref{energie}$_1$ and \eqref{E_zeta} implies that 
$$
E(t)\le E(0)e^{C_1T}+C_2 e^{C_1T}\mathcal{V}\le e^{C_1T}\left(2+C_2\right)\rho_0\qquad \forall\, t\in [\sigma,T]\,,
$$
by which we deduce, in view of \eqref{B_bound}, that
\be
|\bfdelta(t)|+|\theta(t)|\le  e^{\tfrac{C_1T}2}{\left(4+2C_2\right)}^{1/2}\left[\left(\frac{\tau}k\right)^{1/2}+\left(\frac{\varpi}{\rho_1}\right)^{1/2}\right]\sqrt{\rho_0}  =: e^{\tfrac{C_1T}2}\,c_7\sqrt{\rho_0}\qquad \forall\,t\in [\sigma,T]\,,
\eeq{delta-theta}
where $c_7>0$ thus depends on $\tau,k,\varpi, \mathbb{A},\Omega_0$. Therefore, plugging the above bound in \eqref{E_zeta} and using \eqref{3.14} again, we end up with
\be
\frac{d}{dt} E_{\zeta_0}+\frac{1}4\|\mathbb D(\bfu)\|_2^2+{\frac{\zeta_0}\varpi}\bfdelta\cdot\mathbb{B}\cdot\bfdelta+\frac{\zeta_0 k}{\tau}|\theta|^2\le c_6(|\dot{V}|^2+|V|^2 )\,,
\eeq{E_zeta_new_2}
where
\be
\zeta_0:=\frac{1}{8 c_8}\rho_0^{-\tfrac{1}{2}}\,e^{-\tfrac{C_1T}2}\,,
\eeq{zeta3}
and $c_8>0$ depends on $\Omega_0, R,\lambda,\tau,k,\varpi, \mathbb{A}$. Using \eqref{3.14}-\eqref{E_zeta_dis}-\eqref{zeta3} and the Poincar\'e inequality on $\Omega_R$, i.e. $\|\nabla\bfu\|_2\ge c_9R^{-1}\|\bfu\|_2$, where $c_9>0$, we conclude that
\be
\begin{aligned}
&	\frac{1}{4}\|\mathbb{D}(\bfu)\|_2^2+\zeta_0{\frac1\varpi}\bfdelta\cdot\mathbb{B}\cdot\bfdelta+\zeta_0\frac{ k}{\tau}|\theta|^2\\&\ge\frac{\zeta_0}{2}\left(c_8e^{\tfrac{C_1T}2}\rho_0^{\tfrac{1}2}\left(\frac{c_9^2}{R^2}\|\bfu\|^2_2+\frac{\kappa\varpi}{2}\frac{|\bfxi|^2}{\varpi}+\frac{\kappa \tau}{2}\frac{|\dot{\theta}|^2}{\tau}\right)+2\frac{1}{\varpi}\bfdelta\cdot\mathbb{B}\cdot\bfdelta+2\frac{k|\theta|^2}\tau\right)\ge\zeta_0 E\ge \frac{2}{3}\zeta_0\ E_{\zeta_0}\,,
\end{aligned}
\eeq{PC}
provided that
$$
\rho_0\ge \frac{1}{(c_8)^2}\left(\frac{R^4}{c_9^4}+\frac{4}{\kappa^2\varpi^2}+\frac{4}{\kappa^2\tau^2}\right)\,.
$$
Replacing \eqref{PC} in \eqref{E_zeta_new_2} entails that 
\be
\frac{d}{dt}E_{\zeta_0}+\frac{2\zeta_0}{3}E_{\zeta_0}\le c_{10}(|\dot{V}|^2+|V|^2)\qquad\forall \,t\in [\sigma,T]\,,
\eeq{integraE_zeta3}
with $c_{10}>0$ depending on $R,\Omega_0,\varpi,\tau,\mathscr{B},\alpha,\mathbb{A},k$. Integrating the above inequality between $\sigma$ and $t\in (\sigma,T]$, we deduce that
\be
E_{\zeta_0}(t)\le E_{\zeta_0}(\sigma)\exp\left(\tfrac{2}3\zeta_0(\sigma-t)\right)+c_{10}\mathcal{V}\,.
\eeq{Ezet0}
Since, by Lemma \ref{Lemma_glob_ex}, we easily show that 
\be
\lim_{\sigma\to0} E_{\zeta_0}(\sigma)=E_{\zeta_0}(0)\,,
\eeq{E_0}
from \eqref{Ezet0} we deduce
$$
E_{\zeta_0}(t)\le E_{\zeta_0}(0)\exp\left(-\tfrac{2}3\zeta_0 t \right)+c_{10}\mathcal{V}\,,\ \ \mbox{for all $t\in [0,T]$}\,. 
$$
From this it follows that $E_{\zeta_0}(T)\le \rho_0$  if $\rho_0$ satisfies 
\be
\rho_0\ge \frac{c_{10}\mathcal{V}}{1-\exp(-\tfrac{2}3\zeta_0T)}\,.
\eeq{condition_rho0}
Condition \eqref{condition_rho0} is guaranteed provided that $\rho_0$ is greater than some quantity depending on $T,\mathcal{V}, c_8,c_{10}$. Indeed, if we put $x:=T/(12c_8) \rho_0^{-\tfrac{1}2}e^{-\tfrac{C_1T}2}$, in view of \eqref{zeta3}, we have that \eqref{condition_rho0} is equivalent to 
$$
\frac{x^2}{1-e^{-x}}\le \frac{1}{c_{10}\mathcal{V}}\frac{e^{-C_1T}T^2}{(12c_8)^2}\,,
$$
which turns out to be true if $x$ is less than some quantity depending on $T,\mathcal{V}, c_8,c_{10}$. Finally, \eqref{claim2} is proved by integrating \eqref{integraE_zeta3} and using \eqref{E_0}. \end{proof}
\section{Proof of Theorem \ref{main:th}}\label{sec:PMR}
In this section we shall furnish a proof of our main result. It will develop in two steps. In the first, we establish the existence of weak $T$-periodic solutions in the bounded domain $\Omega_R$, for any arbitrary and sufficiently large $R$, with corresponding relevant estimates involving constants independent of $R$. In the second step, we combine this result with the classical ``invading domain" procedure that will lead to a complete proof of Theorem \ref{main:th}.  
\subsection{Step $1$: Existence of $T$-periodic weak solutions in $\Omega_R$}\label{Sec:VeNa}
Thanks to the results established in Lemmas \ref{Lemma_glob_ex}  and \ref{Lemma_dissipation}, we are in a position to prove the existence of a $T$-periodic weak solution with corresponding uniform estimates in the bounded domain $\Omega_R$: 
\smallskip\par\noindent
\be\left\{\ba{l}\medskip\left.\ba{r}\medskip
\lambda\left(\partial_t\bfu+(\bfu-\bsfV)\cdot\nabla\bfu+\dot{\theta}\bfu^\perp\right)= \Div\mathbb T(\bfu,p)\\
\Div\bfu=0\ea\right\}\ \ \mbox{in $\Omega_R\times\real$}\,,\\ \medskip
\ \ \bfu({\bfx},t)=\bsfV({\bfx},t):=\bfxi(t)+\omega(t){\bfx}^\perp\,, \ \mbox{at $\partial\Omega_0\times\real$}\,;\ \\ \medskip
\ \ \bfu(\bfx,t)={\bf0}\,,\ \mbox{at $\partial B_R\times\real$}
\\ \bfu(\bfx,0)=\bfu_0\ \mbox{in $\Omega_R$}
\\
\medskip\left.\ba{r}\medskip
\dot{\bfxi}+\dot\theta\bfxi^\perp+\mathbb B(t)\cdot\bfdelta+\varpi\Int{\partial\Omega_R}{}\mathbb T(\bfu,p)\cdot\bfn= c\, \dot V(t) \bfb_\alpha(t)
\\ \medskip
\dot{\bfdelta}+\dot\theta\bfdelta^\perp={\bfxi-V(t)\bfb_\alpha(t)}\\ \medskip
\dot{\omega}+k\theta+\tau\,\bfe_1\cdot\Int{\partial\Omega_R}{}\bfx\times\mathbb T(\bfu,{ p})\cdot\bfn=d\, \dot V(t)\, \\
\medskip
\dot{\theta}={\omega}
\ea\right\}\ \ \mbox{in $\real$}\,.

\ea\right.
\eeq{01-translation+rotation_bounded}
Such weak solution is defined exactly as in Definition \ref{definition_weak}, once we replace $\mathbb{R}^3$ and $\Omega$ by $B_R$ and $\Omega_R$ respectively, that is:
\begin{definition}\label{definition_weak_OmegaR}
	The quintuple $(\bfu,\bfxi,\bfdelta,\omega,\theta)$ is a {\textit{$T$-periodic weak solution}} to \eqref{01-translation+rotation_bounded} if 
	\begin{enumerate}[(i)]
		\item $\bfu\in L^2(0,T;\mathcal{D}^{1,2}(B_R))$, with $\bfu(\bfx,t)|_{\partial \Omega}=\bfxi(t)+\omega(t)\bfx^\perp$, a.a. $t\in [0,T]$, $\bfxi\in L^2(0,T;\mathbb{R}^3)$, $\omega\in L^2(0,T;\mathbb{R})$.
		\item $\bfdelta\in H^1(0,T;\mathbb{R}^3)$, $\theta\in H^1(0,T;\mathbb{R})$;
		\item $(\bfu,\bfxi,\bfdelta,\omega,\theta)$ satisfies the following equations (with $(\cdot,\cdot)\equiv(\cdot,\cdot)_{\Omega_R}$ and $\langle\cdot,\cdot\rangle\equiv \langle\cdot,\cdot\rangle_{B_R} $)
		\be
		\begin{aligned}
			\int_0^T&\bigg[-\langle\bfu,\bfphi_t\rangle+\lambda((\bfu-\bsfV)\cdot\nabla\bfu+\dot{\theta}\bfu^{\perp},\bfphi)+{\frac1\varpi}\dot{\theta}\bfxi^{\perp}\cdot\hat{\bfrho}\\&
			+2(\mathbb{D}(\bfu),\mathbb{D}(\bfphi))
			+{\frac1\varpi}\hat{\bfrho}\cdot\mathbb{B}\cdot\bfdelta+{\frac{\hat{\alpha}k\theta}\tau}-{\frac1\varpi}\hat{\bfrho}\cdot c\dot{V}{\bfb}_\alpha-{\frac{\hat{\alpha}d\dot{{ V}}}\tau}\bigg]\,{\rm d}t=0\,,\\&\dot{\bfdelta}+\dot\theta\bfdelta^\perp={\bfxi-V(t)\bfb_\alpha(t)}\,,\quad\int_0^T(\bfxi(t)-V(t){\bfb}_\alpha(t)-\dot{\theta}(t)\bfdelta^\perp(t))\,{\rm d}t =0\,,\\&\dot{\theta}=\omega,\quad \int^T_0\omega(t)\,{\rm d}t=0\,,
		\end{aligned}
	\eeq{periodic_weak_bounded}
		for all $\bfphi\in\mathcal{C}_{\sharp}(B_R)$.
	\end{enumerate}
\end{definition}
We are then ready to prove the following existence result.
\begin{proposition}\label{prop:weak_per_bounded}
	Let $V\in H^1(0,T;\mathbb{R})$. Then, for any $R>3R_*$, there is at least one $T$-periodic weak solution to \eqref{01-translation+rotation_bounded}. This solution satisfies the estimates 

\be
\begin{array}{ll}\smallskip
\|\bfxi\|_{L^2}^2+\|\omega\|_{L^2}^2 + \Int0T\|\nabla \bfu(t)\|_{2,\Omega_R}^2\,{\rm d}t\le C\,\mathcal{V}\,,
\\ \|\bfdelta\|_{L^2}^2+\|\theta\|_{L^2}^2 \le C\mathcal{V}\,\left(T^2+{\mathcal{V}}+\Frac{\mathcal{V}^2}{T^{1/2}}+\Frac{\mathcal{V}^3}T\right)\,,
\end{array}
	\eeq{bound_Rindependent}
and 
\be
\|\bfdelta\|_{L^\infty}+\|\theta\|_{L^\infty} \le 
C\mathcal{V}^{1/2} \left(T^{1/2}+T\mathcal{V}^{1/2}+\mathcal{V}+\frac{\mathcal{V}^{3/2}}{T^{1/4}}+\frac{\mathcal{V}^{1/2}+\mathcal{V}^{2}}{T^{1/2}}+\frac{\mathcal{V}}{T^{3/4}}+\frac{\mathcal{V}^{3/2}}T\right)\,.
\eeq{bound_point_two}
	with $C>0$ independent of $R$ and $T$. Moreover, given $R_0>3R_*$, there exists $C_1>0$ depending on $R_0$, but independent of $R$, such that
	\be
	\begin{aligned}
			\int_0^T\|\bfu(t)\|^2_{2,\Omega_{R_0}}\le C_1 \mathcal{V}\,, \qquad \left\|\frac{\partial \bfu}{\partial t}\right\|_{L^1(0,T;\mathcal{D}_0^{-1,2}(\Omega_{R_0}))}\le C_1\left({\mathcal{V}}^{1/2}+\mathcal{V}\right)\,.
	\end{aligned}
	\eeq{bound_2}
\end{proposition}
\begin{proof}
With the exception of \eqref{bound_Rindependent}$_2$ and \eqref{bound_point_two}, the proof of the other statements follows by arguing exactly as in the proof of \cite[Proposition 3.1]{BoGa}. We shall therefore omit it.
In order to show \eqref{bound_Rindependent}$_2$, we begin to set 
	$$
	\bar{w}:=\frac{1}{T}\int_0^Tw(t)\,{\rm d}t\,\qquad \forall\,w \in H^1(0,T)\,.
	$$
Then,  we introduce
	$$
	\begin{aligned}
		&{\sf G}:=-\tfrac{x_3^2}{2}\bar{\theta}\bfe_1+x_2x_1\bar{\theta}\bfe_2\,,\\& {\sf I}:=(\cos \theta\bar{\chi}_{2}+\sin\theta \bar{\chi}_{3})x_3\bfe_1+(-\sin \theta\bar{\chi}_{2}+\cos\theta \bar{\chi}_{3})x_1\bfe_2+
		x_2\bar{\chi}_{1}\bfe_3\,.
	\end{aligned}
	$$
where   $\bfchi:=\mathbb{Q}(\theta)\cdot\bfdelta$.
It is thus clear that the vector fields
	$$
	\bfG(\bfx):=\text{curl}(\psi(|\bfx|){\sf G}(\bfx)) \,,\qquad \bfI(\bfx,t):=\text{curl}(\psi(|\bfx|){\sf I}(\bfx,t))\qquad \forall\,(\bfx,t)\in\mathbb{R}^3\times(0,\infty)
	$$
with  $\psi$ the cut-off function given in Section \ref{Sec:Haraux_trick}, 
can be taken as  test functions in the weak formulation \eqref{periodic_weak_bounded}. Furthermore, they satisfy
	\be
	\begin{aligned}
		&\text{supp }\bfG\subset B_{2R_*}\,,\quad\bfG(\bfx)=\bar{\theta}\bfx^\perp \quad \bfx\in \partial \Omega_0\,,	\\&
		\text{div} \,\bfG(\bfx)=0 \quad \bfx\in \Omega_{R}\,, \quad \sup_{\bfx\in \Omega_{R}}|\bfG(\bfx)|\le c_1|\bar{\theta}|\,,
	\end{aligned}
	\eeq{stime_G}
	and, for all fixed $t\in (0,\infty)$
	\be
	\begin{aligned}
		&\text{supp }\bfI\subset B_{2R_*}\,,\quad\bfI(\bfx,t)=\mathbb{Q}(\theta)^\top\cdot \bar{\bfchi} \quad \bfx\in \partial \Omega_0, 	\\
		\text{div} \,\bfI(\bfx,t)=0 &\quad \bfx\in \Omega_{R}\,, \quad \sup_{\bfx\in \Omega_{R}}|\bfI(\bfx,t)|\le c_1|\bar{\bfchi}|\,,\quad \sup_{\bfx\in \Omega_{R}}|\partial_t\bfI(\bfx,t)|\le c_1|\dot{\theta}||\bar{\bfchi}|\,,
	\end{aligned}
	\eeq{stime_I}
	for some some $c_1>0$ depending on $R_*$. Choosing $\bfphi=\bfG$ in \eqref{periodic_weak_bounded}$_1$, we infer that 
	$$
	\begin{aligned}
		\int_0^T&\bigg[\lambda((\bfu-\bsfV)\cdot\nabla\bfu+\dot{\theta}\bfu^{\perp},\bfG)
		+2(\mathbb{D}(\bfu),\mathbb{D}(\bfG))+{\frac{k\bar{\theta}\theta}\tau}+{\frac{\bar{\theta}d\dot{{ V}}}\tau}\bigg]\,{\rm d}t=0\,,
	\end{aligned}
	$$
	which in turn, recalling the definition of $\bar{\theta}$, furnishes
	\be
	{\frac{k|\bar{\theta}|^2T}{\tau}}=-{\frac{\bar{\theta}d}\tau}\int_0^T\dot{V}\,{\rm d}t-\int_0^T\left[2(\mathbb{D}(\bfu),\mathbb{D}(\bfG))+\lambda((\bfu-\bsfV)\cdot\nabla\bfu+\dot{\theta}\bfu^{\perp},\bfG)\right]\,{\rm d}t\,.
	\eeq{thetabar}
	We can then estimate the right-hand side of \eqref{thetabar}, using H\"older inequality, \eqref{3.14}-\eqref{bound_Rindependent}-\eqref{stime_G}, and the Poincar\'e inequality in $B_{2R*}$, so to have 
	$$
	\begin{aligned}
		{\frac{k|\bar{\theta}|^2T}{\tau}} &\le|\bar{\theta}| \mathcal{V}\left({\frac{d}\tau}T^{1/2}+CT^{1/2}+C\lambda \mathcal{V}\right)
	\end{aligned}
	$$
	for some $C>0$ independent of $R$. This implies that 
	\be
	|\bar{\theta}|\le C{\mathcal{V}}\left(\frac{1}{T^{1/2}}+\frac{ \mathcal{V}}{T}\right)\,.
	\eeq{theta_barC1}
 On the other hand, from Poincar\'e-Wirtinger inequality and \eqref{bound_Rindependent}$_1$ we get
$$
\int_0^T|\theta|^2\,{\rm d}t-|\bar{\theta}|^2T=\int_0^T|\theta-\bar{\theta}|^2\,{\rm d}t\le T^2\int_0^T|\omega|^2\,{\rm d}t\le T^2 C\mathcal{V}\,,
$$
which together with \eqref{theta_barC1} gives the bound for $\theta$ in \eqref{bound_Rindependent}$_2$.  

The bound for  $\bfdelta$ is derived in a similar way. We choose $\bfphi=\bfI$ in \eqref{periodic_weak_bounded}$_1$ and  obtain 
	\be
	\begin{aligned}
		\int_0^T&\bigg[-\lambda(\bfu,\partial_t\bfI)+\lambda((\bfu-\bsfV)\cdot\nabla\bfu+\dot{\theta}\bfu^{\perp},\bfI)
		+2(\mathbb{D}(\bfu),\mathbb{D}(\bfI))
		+{\frac1\varpi}\mathbb{Q}^\top(\theta)\cdot\bar{\bfchi}\cdot\mathbb{B}\cdot\bfdelta\\&\quad+{\frac{c\,\dot{V}}\varpi}\mathbb{Q}^\top(\theta)\cdot\bar{\bfchi}\cdot {\bfb}_{\alpha}\bigg]\,{\rm d}t=0\,,
	\end{aligned}
	\eeq{id_chi}
	where we have used the identity 
	$$
	\dot{\theta}\bfa^\perp\cdot (\mathbb{Q}^\top(\theta)\cdot\bfb)=\bfa\cdot\frac{d}{dt}(\mathbb{Q}^\top(\theta))\cdot \bfb
	$$
	for any $\bfa,\bfb\in\mathbb{R}^3$. Keeping in mind the definition of $\mathbb{B}$, ${\bfb}_\alpha$ and $\bfchi$  from \eqref{id_chi}, we infer that
	$$
	\begin{aligned}
		{\frac{T}\varpi}\bar{\bfchi}\cdot\mathbb{A}\cdot \bar{\bfchi} =&-{\frac{c}\varpi}\left(\int_0^T\dot{V}\,{\rm d}t\right)\, \bar{\bfchi}\cdot \bfb_{\alpha}+\int_0^T\lambda(\bfu,\partial_t\bfI)\,{\rm d}t
		\\&-\int_0^T \left[2(\mathbb{D}(\bfu),\mathbb{D}(\bfI))+\lambda((\bfu-\bsfV)\cdot\nabla\bfu+\dot{\theta}\bfu^{\perp},\bfI)\right]\,{\rm d}t\,.
	\end{aligned}
	$$
By the same arguments used to estimate the right-hand side of \eqref{thetabar}, we have
	$$
	\begin{aligned}
		{\frac{T\rho}\varpi}|\bar{\bfchi}|^2&\le|\bar{\bfchi}|\mathcal{V} \left({\frac{d}\varpi}T^{1/2}+CT^{1/2}+C\lambda \mathcal{V}\right)\,,
	\end{aligned}
	$$
so that $\bar{\bfchi}$ satisfies an estimate similar to \eqref{theta_barC1}.  We next observe that
$$
\dot{\bfchi}=\dot{\mathbb Q}\cdot\bfdelta+\mathbb Q\cdot\dot{\bfdelta}=\mathbb Q\cdot\big[ (\mathbb Q^\top\cdot\dot{\mathbb Q})\cdot\bfdelta+\dot{\bfdelta}\big]\,.
$$
Since, by \eqref{Q}, it results that $(\mathbb Q^\top\cdot\dot{\mathbb Q})\cdot\bfdelta=\dot{\theta}\bfdelta^\perp$, from the previous relation,  \eqref{A}$_1$ and \eqref{periodic_weak_bounded}$_2$, we deduce $\dot{\bfchi}={\bfgamma}-V\bfb_\alpha$, with
	$
	\bfgamma(t)=\mathbb{Q}(\theta(t))\cdot\bfxi(t)\,.
	$ Since $\mathbb{Q}$ is an isometry, we have that $|\bfchi|^2=|\bfdelta|^2$ and $|\bfgamma|^2=|\bfxi|^2$. Then, the estimate of $\bfdelta$ in \eqref{bound_Rindependent}$_2$ follows in the same way as the one for $\theta$.
\par 
We shall next show \eqref{bound_point_two}. We prove only the estimate of $\bfdelta$, as the analogous bound on $\theta$ can be obtain from an entirely similar argument. We begin to observe that 
	$$
	\bfdelta(t)=\bfdelta(s)+\int_s^t\dot{\bfdelta}(\tau)\,d\tau \qquad \forall\, s,t\in [0,T]\,.
	$$  
	Thus, integrating with respect to the variable $s$ between $0$ and $T$ and using Schwarz inequality, we infer that 
	\be
	|\bfdelta(t)|\le \frac{1}T\int_0^T|\bfdelta(s)|\,ds+\int_0^T|\dot{\bfdelta}(\tau)|\,d\tau\le \frac{1}{T^{1/2}}\|\bfdelta\|_{L^2(0,T)}+ \int_0^T|\dot{\bfdelta}(\tau)|\,d\tau\,.
	\eeq{point_b}
	By \eqref{periodic_weak_bounded}$_3$ and a repeated use of Schwarz inequality, we get 
	$$
	\int_0^T|\dot{\bfdelta}(\tau)|\,d\tau\le \|\omega\|_{L^2(0,T)}\|\bfdelta\|_{L^2(0,T)}+T^{1/2}\|\bfxi\|_{L^2(0,T)}+T^{1/2}\mathcal{V}\,.
	$$
	Thus, plugging the above bound in the right-hand side of  \eqref{point_b} and employing \eqref{bound_Rindependent} yields \eqref{bound_point_two}.
\end{proof}

\subsection{Step $2$: Invading domains procedure and proof of Theorem \ref{main:th}}\label{Sec8}
The results accomplished in the previous subsection allow us to give a complete proof of Theorem \ref{main:th}. Indeed, let
 $\{\Omega_n\equiv\Omega_{R_n}\}$, $R_1>3R_*$, be a sequence of domains ``invading'' $\Omega$, that is 
$$
\Omega_n\subset\Omega_{n+1}\,, \quad n\in\mathbb{N}\,;\quad \bigcup_{n=1}^\infty=\Omega\,.
$$
By Proposition \ref{prop:weak_per_bounded}, on each $\Omega_n$ we infer the existence of a sequence of $T$-periodic weak solutions $\{\mathsf{p}_n\}:=\{\bfu_n,\bfxi_n,\bfdelta_n,\omega_n,\theta_n\}$ to \eqref{01-translation+rotation_bounded}. Our objective is to show that, in the limit $R_n\to\infty$, a subsequence of $\{\mathsf{p}_n\}$ converges (in suitable topology) to a weak solution in the sense of Definition \ref{definition_weak}, satisfying all properties stated in Theorem \ref{main:th}. 
For each $n$, we extend $\bfu_n$ by $\bf0$ outside $\Omega_n$, keeping the notation $\bfu_n$ for the extension. By \cite[Remark 2.1]{BoGa}, we have, on the one hand,
$$
\bfu_n(t)\in H^1(\Omega)\cap \mathcal{D}^{1,2}(\mathbb{R}^3) \qquad \text{for a.a.}\,\, t\in [0,T]\,,
$$
and, on the other hand, $\{\bfu_n,\bfxi_n,\omega_n\}$ still satisfies \eqref{bound_Rindependent}-\eqref{bound_2}. Thus, there exists a triple
$$(\bfu,\bfxi,\omega)\in L^2(0,T;\mathcal{D}^{1,2}(\mathbb{R}^3))\times L^2(0,T;\mathbb{R}^3)\times  L^2(0,T;\mathbb{R})
$$ 
such that (up to the choice of a subsequence)
\be
\begin{aligned}
&	{\bfu_n\to \bfu \qquad \text{weakly in }\,L^2(0,T;\mathcal{D}^{1,2}(\mathbb{R}^3))\,,}
\\&\bfxi_n\to \bfxi \qquad \text{weakly in }\,L^2(0,T;\mathbb{R}^3)\,, \\ &\omega_n\to \omega \qquad \text{weakly in }\,L^2(0,T;\mathbb{R})\,.
\end{aligned}
\eeq{convergenze}
In view of \eqref{bound_Rindependent}$_1$, this furnishes, in particular, that $(\bfu,\bfxi,\omega)$ satisfy \eqref{bound_weaksol}$_1$.
Since $\dot{\theta}_n=\omega_n$, the bounds in \eqref{bound_Rindependent} and the compact embedding $H^1(0,T;\mathbb{R})\subset L^2(0,T;\mathbb{R})$  imply that we can extract a further subsequence still denoted by $(\theta_n)_n$ for simplicity such that
\be
\theta_n\to \theta \qquad \text{strongly in }\,L^2(0,T;\mathbb{R})\,.
\eeq{strong_theta}
Setting $\bfchi_n=\mathbb{Q}(\theta_n)\cdot \bfdelta_n$ and 
$\bfgamma_n=\mathbb{Q}(\theta_n)\cdot\bfxi_n$, and proceeding as in the last part of the proof of Proposition \ref{prop:weak_per_bounded}, we show $\dot{\bfchi}_n=\bfgamma_n-V\bfb_\alpha$. Recall also that since $\mathbb{Q}$ is an isometry, it holds
$$
|\bfgamma_n|^2=|\bfxi_n|^2\,,\qquad |\bfchi_n|^2=|\bfdelta_n|^2.
$$
Then, the bounds in \eqref{bound_Rindependent} and the embedding $H^1(0,T;\mathbb{R}^3)\subset C([0,T];\mathbb{R}^3)$  imply that $\|\bfchi_n\|_{L^\infty(0,T)}$ is bounded independently of $n$.
Hence, again from \eqref{bound_Rindependent}, we deduce that 
\be
\begin{aligned}
\int_0^T|\dot{\bfdelta}_n|^2&=\int_0^T|\mathbb{Q}^\top(\theta_n)\cdot \dot{\bfchi}_n+\dot{\theta}_n\,\partial_{\theta}\mathbb{Q}^\top(\theta_n)\cdot {\bfchi}_n|^2\le C\left(\int_0^T|\dot{\bfchi}_n|^2+\int_0^T|\dot{\theta}_n|^2|\bfchi_n|^2\right)
\\&\le C\left(\int_0^T|\bfgamma_n|^2+\int_0^T|V|^2+\|\bfchi_n\|^2_{L^\infty(0,T)}\int_0^T|\dot{\theta}_n|^2\right)\le C\,,
\end{aligned}
\eeq{InFe}
with $C>0$ (which might change from line to line) independent of $n$. All the above,  gives us, in particular, 
\be
\bfdelta_n\to \bfdelta \qquad \text{strongly in }\,L^2(0,T;\mathbb{R})\,.
\eeq{strong_delta}
Thus, combining \eqref{strong_theta} and \eqref{strong_delta} with \eqref{bound_Rindependent}$_2$, we infer  that $(\bfdelta,\theta)$ satisfy \eqref{bound_weaksol}$_2$. Further, by \eqref{bound_point_two}, \eqref{InFe}, \eqref{convergenze}$_3$, \eqref{strong_delta}, \eqref{strong_theta} and the  compact embedding $H^1(0,T;\real^3)\subset C([0,T];\real^3)$ again, we obtain that $(\bfdelta,\theta)$ satisfy also \eqref{bound_point}. 
We now want to upgrade the convergences  in \eqref{convergenze}$_2$-\eqref{convergenze}$_3$ from weak to strong.
Fix $R_0>3R_*$ arbitrarily. By \eqref{convergenze}$_1$-\eqref{bound_2} and  Aubin-Lions-Simon theorem \cite{Simo}, we can extract another subsequence, still denoted by $\{\bfu_n\}$ such that 
\be
\bfu_n\to\bfu  \qquad \text{strongly in} \,\,L^2(0,T;L^{2}(\Omega_{R_0}))\,.
\eeq{strong_conv}
Notice that, by the divergence theorem, we have 
$$
\int_{\partial \Omega_{0}}\left|\left(\bfxi-\bfxi_n+({\omega}-{\omega}_n)\bfx^\perp\right)\cdot \bfn\right|^2d\sigma=|(\bfxi-\bfxi_n)\cdot \bfn|^2|\partial \Omega_{0}|+({\omega}-{\omega}_n)^2\int_{\partial \Omega_{0}}|\bfx^\perp\cdot\bfn|^2\,d\sigma\,.
$$
In view of \cite[Exercise II.4.1]{Gab}, denoting  the trace operator on the inner boundary by $\gamma_0:H^1(\Omega_{R_0})\to L^2(\partial \Omega_{0})$,
we infer that 
\be
C\left(|(\bfxi-\bfxi_n)\cdot \bfn|^2 + (\omega-\omega_n)^2\right)\le \|\gamma_0(\bfu-\bfu_n)\|^2_{2,\partial \Omega_{0}} \le c_{\varepsilon}\|\bfu-\bfu_n\|^2_{2,{\Omega_{R_0}}}+\varepsilon\|\nabla(\bfu-\bfu_n)\|^2_{2}\,,
\eeq{trace_ineq}
for some $C>0$,  arbitrary $\varepsilon>0$, and some $c_\varepsilon>0$. To get a similar estimate for the tangential part of $\bfxi-\bfxi_n$, we write 
\be
\|\gamma_0(\bfu-\bfu_n)\|^2_{2,\partial \Omega_{0}}=|\bfxi-\bfxi_n|^2|\partial \Omega_0|+(\omega-\omega_n)^2\int_{\partial\Omega_{0}}(x_2^2+x_3^2)d\sigma+2(\omega-\omega_n)(\bfxi-\bfxi_n)\cdot\int_{\partial\Omega_{0}}\bfx^\perp\,d\sigma
\eeq{mixed_term}
and we observe that Young's inequality implies
\be
\begin{aligned}
|\bfxi-\bfxi_n|^2|\partial \Omega_0|\le \|\gamma_0(\bfu-\bfu_n)\|^2_{2,\partial \Omega_{0}}+(\omega-\omega_n)^2\left(\frac{2}{|\partial \Omega_0|}-1\right) \int_{\partial\Omega_{0}}(x_2^2+x_3^2)d\sigma\,,
\end{aligned}
\eeq{mixed}
which, recalling \eqref{trace_ineq}, yields
\be
C|\bfxi-\bfxi_n|^2\le c_{\varepsilon}\|\bfu-\bfu_n\|^2_{2,{\Omega_{R_0}}}+\varepsilon\|\nabla(\bfu-\bfu_n)\|^2_{2}\,,
\eeq{trace_ineq2}
again for some $C>0$, arbitrary $\varepsilon>0$, and some $c_\varepsilon>0$.
Now, using \eqref{convergenze}$_1$-\eqref{strong_conv}, \eqref{trace_ineq}-\eqref{trace_ineq2} we deduce that 
\be
\bfxi_n\to \bfxi \qquad \text{strongly in} \,\,L^2(0,T;\mathbb{R}^3)\,,\qquad {\omega}_n\to{\omega} \qquad \text{strongly in} \,\,L^2(0,T;\mathbb{R})\,.
\eeq{strong_conv_obst}
Since $\overline{\omega_n} = 0$ for all $n\in\mathbb{N}$, 
in view of \eqref{convergenze}$_2$, we deduce that $\omega$ has zero average too. As a consequence of the compactness of the embedding $H^1(0,T;\mathbb{R}^3)\subset C([0,T];\mathbb{R}^3)$ and \eqref{strong_conv_obst}$_2$, 
we also infer that (up to a subsequence)
$$
\int_0^T|\dot{\theta}_n\bfdelta_n^\perp-\dot{\theta}\bfdelta^\perp|^2\,{\rm d}t\le2\|\bfdelta-\bfdelta_n\|^2_\infty\int_0^T|\dot{\theta}|^2\,{\rm d}t+2\|\bfdelta^\perp\|^2_\infty\int_0^T|\omega-\omega_n|^2\,{\rm d}t \to 0\,. 
$$
The same arguments allow us to pass to the limit in the identity 
$$
\int_0^T(\bfxi_n-V{\bfb}_{\alpha,n}-\dot{\theta}_n\bfdelta_n^\perp)\,{\rm d}t=0\,,
$$
and to deduce the periodicity of $\bfdelta$ from the fact that $\bfxi-V{\bfb}_{\alpha}-\dot{\theta}\bfdelta^\perp$ has zero time-average. 

We observe that, since we have used \eqref{strong_conv}, the converging sequences $(\bfxi_n)_n$ and $(\omega_n)_n$ may depend on $R_0$. Restricting to nested subsequences for each  $\{R_n\}$ and choosing a diagonal sequence, up to passing to a further subsequence we have that $\{\bfu_n\}$ converges in $L^2(B_{R_0})$ for \textit{all} $R_0$, and similarly all consequent properties follow. 

Finally, the proof of Theorem \ref{main:th} will be completed once we show that the limiting functions found above satisfy the weak formulation \eqref{periodic_weak}. Since $(\bfu_n,\bfxi_n,\bfdelta_n,\omega_n,\theta_n)$ satisfy \eqref{periodic_weak_bounded} and $R$ is arbitrary, for all sufficiently large $n\in\mathbb{N}$, the sequence ${\sf s}_n$ obeys
$$
	\begin{aligned}
	\int_0^T&\bigg[-\langle\bfu_n,\bfphi_t\rangle+\lambda((\bfu_n-\bsfV_n)\cdot\nabla\bfu_n+\dot{\theta}_n\bfu_n^{\perp},\bfphi)+{\frac1\varpi}\dot{\theta}_n\bfxi^{\perp}_n\cdot\hat{\bfrho}\\&
	+2(\mathbb{D}(\bfu_n),\mathbb{D}(\bfphi))
	+{\frac1\varpi}\hat{\bfrho}\cdot\mathbb{B}\cdot\bfdelta_n+{\frac{\hat{\alpha}k\theta_n}\tau}-{\frac1\varpi}c\hat{\bfrho}\cdot \dot{V}{\bfb}_{\alpha,n}-{\frac{\hat{\alpha}d\dot{{ V}}}\tau}\bigg]\,{\rm d}t=0
\end{aligned}
$$
for all $\bfphi\in\mathcal{C}_{\sharp}(\mathbb{R}^3)$. 
Then the  convergences proved in \eqref{convergenze}-\eqref{strong_theta}-\eqref{strong_delta}-\eqref{strong_conv}-\eqref{strong_conv_obst} allow to see that $\bfu,\bfxi,\bfdelta,\omega, \theta$ satisfy \eqref{bound_weaksol},  and to replace in the above identity $\bfu_n,\bfxi_n,\bfdelta_n,\omega_n$ and $\theta_n$ by $\bfu,\bfxi,\bfdelta,\omega$ and $\theta$ respectively, treating the nonlinear terms as in \cite[Proposition 3.1]{BoGa}. This concludes the proof. 

\section{The Case of a symmetric body}\label{sec:symm}
In this final section, we discuss how the method that we have presented so far in order to prove the existence of a $T$-periodic weak solution to problem \eqref{01} simplifies in the particular case where 
\be
\text{$\mathscr B$ is invariant  with respect to the group of rotations around the axis $\bfe_1$\,.}
\eeq{assumption}
The main difference with respect to the case where the axis $\bfe_1$ is \textit{not} an axis of revolution for $\mathscr B$ is that we only need the first step of the change of variables introduced in Section \ref{sec:body-fixed} to make the fluid domain time-independent. In particular, it is sufficient to attach the frame to the barycenter, 
use the change of variables and unknowns defined in \eqref{change} to obtain that $(\tilde{\bfu},\tilde{p},\tilde\bfchi,\tilde\bfgamma,\theta,\omega)$ satisfy
\be\left\{\ba{l}\medskip\left.\ba{r}\medskip
\partial_t\tilde{\bfu}+\lambda(\tilde{\bfu}-\tilde\bfgamma(t))\cdot\nabla\tilde{\bfu}=\Div \mathbb S(\tilde{\bfu},\tilde{p})\\
\Div\tilde{\bfu}=0\ea\right\}\ \ \mbox{in $\Omega\times\mathbb{R}$}\,,\\ \medskip
\ \ \bfw(\tilde{\bfx},t)=\tilde\bfgamma(t)+\omega(t)\tilde{\bfx}^\perp\,, \ \mbox{at $\partial\Omega_0\times\mathbb{R}$}\,;\ \\ \medskip
\ \ \Lim{|\tilde{\bfx}|\to\infty}\bfw(\tilde{\bfx},t)={\bf{0}}\,, \ \ t\in\real\,,\\
\medskip\left.\ba{r}\medskip
\dot{{\tilde\bfgamma}}+{\mathbb A}\cdot\tilde\bfchi+\varpi\Int{\partial\Omega_0}{} \mathbb S(\tilde{\bfu},\tilde{p})\cdot\tilde\bfn=c\,\dot{V}(t)\tilde\bfb_\alpha
\\ \medskip
\dot{\tilde\bfchi}={\tilde\bfgamma-V(t)\tilde\bfb_\alpha}\\ \medskip
\dot{\omega}+k\theta+\tau\,\bfe_1\cdot\Int{\partial\Omega_0}{}\tilde{\bfx}\times \mathbb S(\tilde{\bfu},\tilde{ p})\cdot\tilde\bfn=d\,\dot{V}(t)\, \\
\medskip
\dot{\theta}={\omega}
\ea\right\}\ \ \mbox{in $\real$\,,}
\ea\right.
\eeq{01_sym}
which is set on the fixed fluid domain $\Omega$. 
With respect to problem \eqref{01-translation+rotation_or}, we observe that the displacement of the barycenter $\tilde\bfchi$ and the rotation $\theta$ remain decoupled in \eqref{01_sym}. This makes problem \eqref{01_sym} treatable exactly as in \cite{BoGa}. On the other hand, if assumption \eqref{assumption} holds, we also remark that 
	$$
	\int_{\partial \Omega_0}\tilde{\bfx}^\perp\,d\sigma=0\,.
	$$
We can then apply directly the trace inequality in \cite[Exercise II.4.1]{Gab} to upgrade the convergences of the body velocities from weak to strong as in Section \ref{Sec8}, since no mixed term as in  \eqref{mixed_term} do appear.
\bigskip\par
\noindent
{\bf Acknowledgements.} The work of G.P.~Galdi is partially supported by National Science Foundation Grant DMS-2307811.
D. Bonheure and C. Patriarca are supported by the WBI grant ARC Advanced 2020-25 ``PDEs in interaction'' at ULB. D. Bonheure is also partially supported by the Francqui Foundation as Francqui Research Professor 2021-24, the FNRS PDR grant  T.0020.25 and the Fonds Thelam 2024-F2150080-0021313. 

\medskip
%

\ed